\documentclass[letter,11pt]{article}

\usepackage[linktocpage=true]{hyperref}
\usepackage[margin=1.2 in, top=1 in, bottom= 1.2 in]{geometry}

\usepackage[normalem]{ulem}

\usepackage{lipsum}
\newcommand\blfootnote[1]{%
  \begingroup
  \renewcommand\thefootnote{}\footnote{#1}%
  \addtocounter{footnote}{-1}%
  \endgroup
}

\makeatletter
\g@addto@macro\normalsize{
  \setlength\abovedisplayskip{8pt}
  \setlength\belowdisplayskip{8pt}
  \setlength\abovedisplayshortskip{8pt}
  \setlength\belowdisplayshortskip{8pt}
  }
\makeatother

\usepackage{tocloft}

\usepackage[english]{babel}
\usepackage{amsfonts}
\usepackage{mathrsfs}
\usepackage{bbm}
\usepackage{latexsym}
\usepackage{amssymb}
\usepackage{mathtools}

\usepackage{enumitem}
\setlist{nolistsep}
\usepackage{amsthm}
\usepackage[capitalize]{cleveref}
\crefformat{footnote}{#2\footnotemark[#1]#3} 

\newcommand\eqnitem[1][]{%
  \ifx\relax#1\relax  \item \else \item[#1] \fi
  \abovedisplayskip=0pt\abovedisplayshortskip=0pt~\vspace*{-\baselineskip}}

\newtheoremstyle{plain}{3mm}{3mm}{\slshape}{}{\bfseries}{.}{.5em}{}
\newtheoremstyle{definition}{2mm}{2mm}{}{}{\bfseries}{.}{.5em}{}
\theoremstyle{plain}
	
\newtheorem{theorem}{Theorem}

\newtheorem{lemma}[theorem]{Lemma}

\newtheorem{corollary}[theorem]{Corollary}

\newtheorem{question}[theorem]{Question}
\theoremstyle{definition}
\newtheorem{definition}[theorem]{Definition}

\newtheorem{example}[theorem]{Example}

\theoremstyle{plain}
\newcounter{MainTheoremCounter}

\newtheorem{Maintheorem}[MainTheoremCounter]{Theorem}

\theoremstyle{plain}
\newtheorem*{namedthm}{\namedthmname}
\newcounter{namedthm}
	

\usepackage{chngcntr}
\counterwithin{theorem}{section}

\numberwithin{equation}{section}

\allowdisplaybreaks

\usepackage{xcolor}
\usepackage{xcolor}
\definecolor{Color2}{rgb}{0.78, 0.11, 0.0}
\hypersetup{citecolor = black,colorlinks,
			linkcolor = black,
			urlcolor = Color2}
\definecolor{DarkGreen}{RGB}{1,150,32}
\hypersetup{citecolor = black,colorlinks,
			linkcolor = black,
			urlcolor = Color2}

\usepackage{titlesec}
\titleformat{\section}
  {\Large\center\bfseries}
  {\thesection.}{.7em}{}
\titlespacing*{\section}{0pt}{3.5ex plus 0ex minus 0ex}{1.5ex plus 0ex}
\titleformat{\subsection}
  {\large\center\bfseries}
  {\thesubsection.}{.7em}{}
\titlespacing*{\subsection}{0pt}{3.5ex plus 0ex minus 0ex}{1.5ex plus 0ex}
\titleformat{\subsubsection}
  {\center\bfseries}
  {\thesubsubsection.}{.7em}{}
\titlespacing*{\subsubsection}{0pt}{3.5ex plus 0ex minus 0ex}{1.5ex plus 0ex}

\addto\captionsenglish{}

\usepackage{titling}
\usepackage{bibentry}
\nobibliography* 

\usepackage{authblk}

\newcommand{\Cech}{\v{C}ech}

\newcommand{\N}{\mathbb{N}}
\newcommand{\Z}{\mathbb{Z}}

\newcommand{\defeq}{\vcentcolon=}

\newcommand\restr[2]{{ \left.\kern-\nulldelimiterspace #1 \right|_{#2}}}

\usepackage[normalem]{ulem}

\newcommand{\family}{\mathcal{F}}

\newcommand{\RP}{\mathbf{RP}}
\newcommand{\seq}{S_\text{eq}}

\renewcommand{\emptyset}{\varnothing}

\newcommand{\gro}{\text{Gr}}
\newcommand{\igro}{i_{\gro{}}}

\newcommand{\lean}[1]{}
\newcommand{\leanok}{}

\begin{document}

\title{The regionally proximal relation for commutative semigroup actions}

\author[1]{Angelina Blahodatna}
\author[1]{Lauren Detmold}
\author[1]{Daniel Glasscock}
\author[2]{Anh N. Le}

\affil[1]{\small Dept. of Mathematics and Statistics, U. of Massachusetts Lowell, Lowell, MA, USA}
\affil[2]{\small Dept. of Mathematics,
	U. of Denver, Denver, CO, USA}

\date{}

\maketitle



\begin{abstract}
    The regionally proximal and equicontinuous structure relations are fundamental relations in topological dynamics that capture the equicontinuous behavior in a topological dynamical system and its factors. 
    For minimal actions of abelian groups, these relations are known to be equivalence relations that coincide.
    In this paper, we generalize these facts to semigroup actions: for minimal actions of commutative semigroups, the regionally proximal and equicontinuous structure relations are equivalence relations and the two coincide.
    We leverage our approach to prove a refinement: the regionally proximal relation for such actions is equal to an analogue and strengthening of Veech's D relation.
    We also develop the machinery of natural extensions for commutative, not-necessarily-cancellative semigroup actions that act by surjections, concluding that the maximal equicontinuous factor of a minimal action of a commutative semigroup is the same as the maximal equicontinuous factor of its natural extension.
    
    We formally verify all of the results in this paper in Lean.  The main results are registered on Palomar, and we link to a Github repository containing code for the complete verification.
\end{abstract}

\clearpage

\setcounter{tocdepth}{2}
\tableofcontents

\clearpage

\section{Introduction}

In this paper, a \emph{(topological dynamical) system} is an action of a (discrete) semigroup $(S,\cdot)$ on a compact Hausdorff space $X$ by continuous maps.  We call $X$ an \emph{$S$-system} and denote the action of $S$ by juxtaposition: for $s \in S$ and $x \in X$, the image of $x$ under the map $s : X \to X$ is written $sx$.
A \emph{factor} of $X$ is an $S$-system $Y$ together with a continuous surjection $\pi: X \to Y$ that intertwines the actions of $S$, i.e. $s (\pi x) = \pi (s x)$ for all $s \in S$ and $x \in X$.
As shorthand, we write that $\pi: X \to Y$ is a factor of $X$ or an \emph{extension of $Y$}. The systems $X$ and $Y$ are \emph{isomorphic} if there exists a homeomorphism $X \to Y$ that intertwines the actions of $S$.

Equicontinuity and minimality play an important role in our results.
An $S$-system $X$ is \emph{equicontinuous} if the family of maps $\{s : X \to X \mid s \in S\}$ is (uniformly) equicontinuous: for all open sets $\beta \subseteq X^2$ containing the diagonal, there exists an open set $\alpha \subseteq X^2$ containing the diagonal such that for all $s \in S$, $s \alpha \defeq \{(sx, sy) \mid (x,y) \in \alpha\} \subseteq \beta$.  The system $X$ is \emph{minimal} if for all $x \in X$, the set $\{sx \mid s \in S\}$ is dense in $X$.

\subsection{The regionally proximal and equicontinuous structure relations}
\label{sec_rp_and_seq}

The regionally proximal relation is a fundamental relation in topological dynamics that captures equicontinuity in a system and its factors.

\begin{definition}[Regionally proximal relation]
\label{def:rp_forward}
    Let $S$ be a semigroup, $X$ be an $S$-system, and $x, y \in X$.  The point $x$ is \emph{regionally proximal to $y$} if for all open sets $U \subseteq X$ containing $x$, $V \subseteq X$ containing $y$, and $\alpha \subseteq X^2$ containing the diagonal, there exist $x' \in U$, $y' \in V$, and $s \in S$ with
    \[
        (sx', sy') \in \alpha.
    \]
    Equivalently,
    \[
        \RP \defeq \bigcap_\alpha \overline{\bigcup_{s \in S} s^{-1} \alpha},
    \]
    where $s^{-1} \alpha \defeq \{(w,z) \in X^2 \mid (sw, sz) \in \alpha\}$ and where the intersection is over all open sets $\alpha \subseteq X^2$ containing the diagonal.
    Loosely speaking, the point $x$ is regionally proximal to $y$ if there are points near $x$ and $y$ whose images under a map in $S$ are nearby.
    We write $\RP_{X,S} \subseteq X^2$ for the set of pairs $(x,y)$ where $x$ is regionally proximal to $y$.
    To simplify the notation, we write $\RP_X$, $\RP_S$, or $\RP$ when the acting semigroup, system, or both are clear from the context.
\end{definition}

For group actions, the regionally proximal relation has been well-studied and has found numerous applications.
When $S$ is a group and $\pi: X \to Y$ is a factor, it is well-known that:
\begin{enumerate}
    \item $\RP$ is a reflexive, symmetric, closed, $S$-invariant relation on $X$;
    
    \item \label{item_rp_fact_two} the system $X$ is equicontinuous if and only if $\RP$ is contained in the diagonal (cf. \cite[IV. 2.10]{devries_book_1993}); more generally,
    
    \item \label{item_rp_fact_three} if the system $X$ is minimal, the factor $Y$ is equicontinuous if and only if $\RP \subseteq \{(x,y) \in X^2 \mid \pi x = \pi y\}$ (cf. \cite[V. 2.14]{devries_book_1993}); and

    \item \label{item_rp_fact_three_pt_five}
    if the system $X$ is minimal, then $(\pi \times \pi) \RP_X = \RP_Y$ (cf. \cite[V. 2.14]{devries_book_1993}).    
\end{enumerate}
It is important to know when $\RP$ is transitive and, hence, an equivalence relation.  
In this case, the quotient $X / \RP$ is a compact Hausdorff space (cf. \cref{lemma_quotient_is_compact_hausdorff}), the quotient map $X \to X/\RP$ is a factor map of $S$-systems, and, by Fact \ref{item_rp_fact_three}, the system $X/\RP$ is the largest equicontinuous factor of $X$.

Closely related to the regionally proximal relation is the equicontinuous structure relation, $\seq \subseteq X^2$, defined for semigroup actions by $(x,y) \in \seq$ if $\pi x = \pi y$ for all equicontinuous factors $\pi: X \to Y$. 
When $S$ is a group, it is well-known that:
\begin{enumerate}
    \setcounter{enumi}{4}
    \item \label{item_rp_fact_four}
    $\seq$ is a closed, $S$-invariant equivalence relation, and $X / \seq$ is the largest equicontinuous factor of $X$ (cf. \cite[V. 1.3]{devries_book_1993}); and
    
    \item 
    \label{item_rp_fact_five}
    if $S$ is abelian and $X$ is minimal, then $\seq = \RP$; in particular, $\RP$ is an equivalence relation and $X \to X/\RP$ is the largest equicontinuous factor of $X$ (\cite[Thm. 1.1]{veech_1968}).
\end{enumerate}
It follows from Facts \ref{item_rp_fact_three} and \ref{item_rp_fact_four} that $\RP \subseteq \seq$.  We see, then, that equality holds if and only if $\RP$ is an equivalence relation.

While these facts for group actions have been known for more than half a century, the fundamental connection between the regionally proximal relation and the largest equicontinuous factor of a system remains relevant today.
Host, Kra, and Maass \cite{host_kra_maass_2010} defined higher-order regionally proximal relations and connected them to nilfactors, algebraically-structured factors of a system that generalize the group rotations behind the largest equicontinuous factors.
These higher-order variants were subsequently shown by Shao and Ye \cite{shao_ye_2012} to be equivalence relations.
Together, these results have provided the foundation for numerous further developments and applications of topological dynamics over the past two decades.

\subsection{From groups to semigroups}

Several of the facts mentioned in the previous subsection have been extended from group actions to semigroup actions, although sometimes under additional hypotheses. For systems in which elements of the acting semigroup all act surjectively, Dai and Xiao \cite[Thm.~1.5]{dai_xiao_2017} prove that equicontinuity is equivalent to the triviality of the regionally proximal relation, extending Fact~\ref{item_rp_fact_two}; see also the corresponding characterization in \cite[Cor.~1.16]{auslander_dai_2019}.
In the relative setting of extensions, Cao and Dai \cite[Thm.~1.1.5 and~3.1.2]{cao_dai_2022} establish semigroup analogues relating relative regional proximality to relative equicontinuity and prove lifting results for the regionally proximal relation, generalizing Facts 2, 3, and 4.
In the same paper, they also summarize a number of related results by Gerko \cite{gerko_1972,gerko_2001}.

The pursuit of generality to semigroup actions is not for mere taxonomic reasons: many central notions in topological dynamics have natural semigroup analogues, motivated both by the ubiquity of non-invertible dynamics and by applications to the theory of recurrence, combinatorial number theory, and the structure theory of dynamical systems.
Neither is the pursuit entirely straightforward: a number of technical hurdles are introduced when moving from group actions to semigroup actions.

In this paper, we develop elementary techniques to overcome some of those technical hurdles; see, for example, our proof of \cref{lemma_rp_functorial_property_minimal} and the discussion preceding it.  Our methods are purely topological and combinatorial, avoiding invariant measures (as in \cite{mcmahon_1978}).  Our use of ultrafilters allows us to sidestep more complicated techniques involving Ellis groups, the $\tau$-topology, and universal minimal flows (as in \cite{cao_dai_2022}). Hand-formalization of the results in Lean (see \cref{sec_lean}) informed the paths we chose to the main results: we aimed to keep the arguments as elementary as possible, drawing on existing results in Mathlib as much as possible.

\subsection{\texorpdfstring{$\RP$}{RP} is an equivalence relation}
\label{sec_rp_and_seq_for_semi_acts}

Our first main result is a generalization of Fact \ref{item_rp_fact_five} from \cref{sec_rp_and_seq}.

\begin{Maintheorem}
\label{mainthm_RP_is_eq_relation}
\lean{RPisEquiStructureRelation}
    Let $S$ be a commutative semigroup and $X$ be a minimal $S$-system.  The regionally proximal relation coincides with the equicontinuous structure relation:
    \[\RP = \seq.\]
    In particular, $\RP$ is an equivalence relation, and $X \to X/\RP$ is the largest equicontinuous factor of $X$.
\end{Maintheorem}

In unpublished work, Dai \cite[Thm. 1.4]{dai_2019} proves that if a minimal $S$-system possesses an invariant measure, then $\RP = \seq$.
Since commutative semigroups are amenable, \cref{mainthm_RP_is_eq_relation} follows from Dai's result.

Our work differs from Dai's in substantial ways. While Dai's proof uses what is called McMahon pseudo-metrics, our proof goes by way of a modern standard for relating higher-order regionally proximal relations and nilpotent algebraic structure. Specifically, we study the dynamics of cube spaces, following the ideas in Glasner, Gutman, and Ye's recent paper \cite{glasner_gutman_ye_2018}.
Thanks to this modern approach, we expect the proof of \cref{mainthm_RP_is_eq_relation} to be suitable for generalization to the higher-order regionally proximal relations (see \cref{sec_higher_order}).
As mentioned above, for the purpose of simplifying the Lean verification, we use ultrafilters as an alternative to both enveloping semigroups and the theory of nets.

Our approach also allows us to strengthen \cref{mainthm_RP_is_eq_relation} into \cref{maintheorem_thm_a_refinement} below, which is a key ingredient to a combinatorial application of this work in another paper \cite{blahodatna_detmold_glasscock_le_Delta_2026}. (See \cref{sec_semigroups} for the definition of syndetic sets.)

\begin{Maintheorem}
\label{maintheorem_thm_a_refinement}
\label{thm_delta_set_builder_prestep}
\lean{xyInRPImpliesSyndeticVisitTimeIntersection}
\leanok{}
    Let $S$ be a commutative semigroup, $X$ be a minimal $S$-system, and $x, y \in X$. The point $(x,y) \in \RP$ if and only if for all non-empty, open $V, W \subseteq X$ with $y \in V$, the set
    \begin{align}
    \label{eqn_d_relation_set}
        \big\{ s \in S \mid sx \in W \text{ and } sV \cap W \neq \emptyset \big\}
    \end{align}
    is syndetic.
\end{Maintheorem}

This theorem, which we prove in \cref{sec_proof_of_recurrence_corollary}, concerns a refinement of the regionally proximal relation in which only one point in the pair $(x, y)$ is required to be perturbed.
To make this precise, for an $S$-system $X$, define $\mathbf{D} \subseteq X^2$ to be the set of pairs $(x,y)$ such that for all open $V \subseteq X$ containing $y$ and $\alpha \subseteq X^2$ containing the diagonal, there exists $y' \in V$ and $s \in S$ such that $s (x,y') \in \alpha$.
As far as we know, this relation was first considered by Veech \cite{veech_1968}, who called it $\mathbf{D}$, for abelian group actions.
It is clear from the definitions that $\mathbf{D} \subseteq \RP$.  Replacing ``syndetic'' with ``non-empty'' and utilizing minimality, \cref{maintheorem_thm_a_refinement} would say that $\mathbf{D} = \RP$.
For minimal actions of amenable groups and semigroups, that $\mathbf{D} = \RP$ was shown by McMahon \cite{mcmahon_1978} and Dai \cite{dai_2019}, respectively.
The novelty and utility of \cref{maintheorem_thm_a_refinement} is that we show the set in \eqref{eqn_d_relation_set} to be syndetic.

An oft-cited example of McMahon \cite[Ex. 1.8]{mcmahon_1976} (cf. \cite[V(1.8)2]{devries_book_1993}) details a minimal action of the free group on two generators for which $\RP$ is not transitive, demonstrating that we cannot expect the conclusions of the main results in this section to hold without some assumptions on the semigroup $S$.
Given the aforementioned precedent, it seems likely that \cref{maintheorem_thm_a_refinement} could hold for minimal actions of amenable semigroups.
Beyond the class of amenable semigroups, we do not know to what extent the conclusions of Theorems \ref{mainthm_RP_is_eq_relation} and \ref{maintheorem_thm_a_refinement} could be expected to hold.

\subsection{The natural extension and regional proximality}

A homeomorphism $T: X \to X$ of a compact Hausdorff space $X$ generates both an action of $\N$ and an action of $\Z$ on $X$.  With an established notion of regional proximality for both semigroup and group actions, it is natural to ask how the notions differ, if at all, between these actions.  For example, if a pair of points in $X$ are regionally proximal under the $\N$-action, then they are obviously regionally proximal under the $\Z$-action.  But what about the converse?

If $T: X \to X$ is surjective but not invertible, then it only generates an action of $\N$ on $X$ by continuous surjections.  In this case, the natural extension of $X$, $\pi: W \to X$, is an $\N$-system $W$ on which $\N$ acts by homeomorphisms.  If a pair of points in $W$ are regionally proximal, then it is easy to see that the pair of their images under $\pi$ are regionally proximal in $X$.  But what about the converse?  It is a closely related question to ask how the largest equicontinuous factors of $W$ and $X$ compare.

These questions are the most basic cases of ones that we treat in the last part of this paper.  The more general framing begins with a semigroup homomorphism $i: S \to G$ from a semigroup $S$ into a group $G$.  On the one hand, any $G$-system can be considered an $S$-system by defining the action of $s \in S$ to be the action of its image $is$.  In certain cases, it is possible to go in the other direction, starting with an $S$-system and creating from it a $G$-system.  In both cases, the basic question is the same: how does regional proximality under the actions of $S$ and $G$ compare?

Our first main result in this direction shows that the sets of regionally proximal pairs under commuting minimal actions are the same.

\begin{Maintheorem}
\label{mainthm_same_rp_for_comm_actions}
\lean{forTwoMinCommActionsRPsAreSame}
    Let $S$ and $T$ be commutative semigroups.  Let $X$ be both a minimal $S$-system and a minimal $T$-system.  If the actions of $S$ and $T$ commute, then $\RP_S = \RP_T$.
\end{Maintheorem}

We prove \cref{mainthm_same_rp_for_comm_actions} in \cref{sec_proof_of_comm_actions_thm}.
The result is known in the group-action setting; see, for example, \cite[Thm. 1.1]{shao_xu_2026}.
Our proof of \cref{mainthm_same_rp_for_comm_actions} follows their ideas with the necessary adaptations for semigroup actions.
As a first application, we see that if we generate an $S$-system from a $G$-system via the homomorphism $i: S \to G$, regional proximality under the actions of $S$ and $G$ are the same provided that $G$ is abelian and $S$ acts minimally.\\

To address the other questions, we are led to develop the machinery of natural extensions for surjective actions of commutative semigroups $S$.
In this setting, for every $S$-system $X$ we construct a smallest $S$-system extension $\pi: W \to X$ on which $S$ acts by homeomorphisms; see \cref{sec_nat_ext} for the precise definitions.

\begin{Maintheorem}
\lean{natExtOfMinimalCommSystem}
\label{mainthm_nat_ext}
    Let $S$ be a commutative semigroup and $X$ be a minimal $S$-system.
    There exists a minimal $S$-system $W$ on which $S$ acts by homeomorphisms and a factor map $\pi: W \to X$ with the property that for all extensions $\rho: V \to X$ of $S$-systems in which $S$ acts on $V$ by homeomorphisms, there exists a factor map $\xi: V \to W$ such that $\rho = \pi \circ \xi$.
\end{Maintheorem}

The extension in \cref{mainthm_nat_ext} is realized by way of the Grothendieck group $\gro(S)$ and yields few surprises: it satisfies all of the same properties that are enjoyed by the more-familiar natural extension for $\N$-systems.
In the event that $S$ is assumed to be cancellative, the Grothendieck group is the more-familiar group of quotients, and our results match some of those in \cite{briceno2025naturalextensionsembeddablesemigroup, liang_huang_2024}.
That a nice natural extension exists without assuming cancellativity was, to us, surprising.
The secret is that while not every commutative semigroup embeds into a group, every commutative semigroup action by surjections ``embeds'' into an action of a cancellative, commutative semigroup by surjections; see \cref{thm_igro_same_implies_elts_act_same}.

We conclude the paper by addressing the relationship between regional proximality in a system $X$ and in its natural extension $W$.
We show in \cref{thm_rp_of_natural_extension} that two points are regionally proximal in the natural extension $\pi: W \to X$ if and only if their images under $\pi$ are regionally proximal.
This leads quickly to the following theorem, which says, roughly speaking, that no matter how one interprets it, the maximal equicontinuous factor of the natural extension is the same as the maximal equicontinuous factor of the base.
We prove \cref{mainthm_kronecker_is_same_as_extension_kronecker} and its precursors in \cref{sec_proof_of_same_kron}.

\begin{Maintheorem}
\label{mainthm_kronecker_is_same_as_extension_kronecker}
\lean{theoremE}
\leanok{}
    Let $S$ be a commutative semigroup, $X$ be a minimal $S$-system, and $\pi: W \to X$ be its natural extension. The following $S$-systems are the same up to isomorphism:
    \begin{enumerate}
        \item \label{item_x_kron}
        $X/\RP_{X,S}$, the largest equicontinuous factor of the $S$-system $X$;
        \item \label{item_w_kron}
        $W / \RP_{W,S}$, the largest equicontinuous factor of the natural extension $W$ of $X$ as an $S$-system; and
        \item \label{item_w_kron_gp}
        $W / \RP_{W,\gro(S)}$, the largest equicontinuous factor of the natural extension $W$ of $X$ as a $\gro(S)$-system, made into an $S$-system via the homomorphism $\igro: S \to \gro(S)$.
    \end{enumerate}
\end{Maintheorem}

\subsection{Outline}

The paper is organized as follows. In \cref{sec_prelims}, we review the necessary preliminaries. In \cref{sec_rp_for_semigp_actions}, we develop the basic theory of the regionally proximal relation for semigroup actions.
In \cref{sec_corner_space}, we introduce the corner space construction whose dynamics is studied in \cref{sec_corner_space_dynamics}, where Theorems \ref{mainthm_RP_is_eq_relation} and \ref{maintheorem_thm_a_refinement} are proved.
\cref{sec_nat_extension} is devoted to the natural extension, where we prove Theorems \ref{mainthm_same_rp_for_comm_actions}, \ref{mainthm_nat_ext}, and \ref{mainthm_kronecker_is_same_as_extension_kronecker}.

\subsection{AI use disclosure}

Generative AI was not used to generate mathematical results or proofs; all mathematical content is the work of the authors, and the authors take full responsibility for it. Generative AI was used only to assist with literature searches and language editing.
All references suggested by AI tools were independently verified by the authors.

\subsection{Lean verification}
\label{sec_lean}

We verified all of the results in this paper in Lean.  The verification code can be found on Github \cite{leaps_github_2026}.  In addition, we registered the main results appearing in the introduction with Palomar \cite{rp_palomar}.  Most of the verification was coded by hand; the precise use of AI coding agents is explained in the Github readme file and in the Palomar submission in the automation section of the formalization metadata file.

\subsection{Acknowledgments}

The authors gratefully acknowledge support from the US National Science Foundation under a LEAPS Grant, Award Number 2418589.

\section{Preliminaries}
\label{sec_prelims}

Most of this section is standard preparatory material, though there are some novel lemmas that may be of independent interest; see, e.g., \cref{lemma_forward_backward_set_orb_clos_coincide}.

\subsection{Sets}

Throughout this paper, when elements of a set $S$ act by maps with domain a set $X$, we denote by $sx$ the image of $x \in X$ under the map $s \in S$.  For subsets $H \subseteq S$ and $A \subseteq X$, we define
\begin{align*}
    Hx &\defeq \{sx \mid s \in H\},\\
    sA &\defeq \{sx \mid x \in A\},\\
    HA &\defeq \{sx \mid s \in H, \ x \in A\} \\
    s^{-1}A &\defeq \{x \in X \mid sx \in A\},\\
    H^{-1}A &\defeq \{x \in X \mid \exists s \in H, \ sx \in A\}.
\end{align*}
We say that $A$ is \emph{$S$-invariant} if $SA \subseteq A$.  It is easy to see that an arbitrary intersection of $S$-invariant sets is $S$-invariant.

If not specified otherwise, the set $S$ acts on $X^2$ via the diagonal action: $s(x,y) = (sx,sy)$.
The \emph{diagonal of $X^2$} is the set
\[
    \Delta_X \defeq \big\{ (x,x) \mid x \in X \big\},
\]
which we alternatively denote by $\Delta$ when the ambient set $X$ is clear from context.
For maps $f, g: X \to Y$, we denote by $f \times g$ the map $X^2 \to Y^2$ defined by
\[(f \times g) (x_1, x_2) = (f x_1, g x_2).\]

We label the following lemma so that it is easy to recall when we need it.

\begin{lemma}
    \label{lemma_commuting_map_algebra}
    \lean{preimageOfCommutingMapsCommutes}
    \leanok{}
    Let $X$ be a set and $f, g : X \to X$ be commuting maps.  For all $A \subseteq X$, $f g^{-1} A \subseteq g^{-1} f A$.
\end{lemma}

\begin{proof}
\leanok{}
    This is an easy exercise left to the reader.
\end{proof}

\subsection{Topology}

We work throughout the paper with compact Hausdorff spaces.  While we avoid the terminology around uniform spaces, it should be noted that the class of compact Hausdorff spaces is the same as the class of compact uniform spaces.  Indeed, compact Hausdorff spaces have a unique uniformity given by sets whose interior contains the diagonal, while compact uniform spaces are compact and completely regular, hence Hausdorff.

The following lemmas are used repeatedly throughout the paper.

\begin{lemma}
\label{lemma_closed_nbhd_prod_in_diag}
\lean{openClosureProductInEntourage}
\leanok{}
    Let $X$ be a compact Hausdorff space.  For all $x \in X$ and all open $\alpha \subseteq X^2$ containing the diagonal, there exists an open set $U \subseteq X$ containing $x$ such that $\overline{U} \times \overline{U} \subseteq \alpha$.
\end{lemma}

\begin{proof}
\leanok{}
    This is a simple exercise using the definition of the product topology and the fact that compact Hausdorff spaces are regular.
\end{proof}

\begin{lemma}
\label{lemma_image_closure_equals_closure_image}
 \lean{imageClosureIsClosureImage}
 \leanok{}
    Let $f: X \to Y$ be a continuous map of compact Hausdorff spaces.  For all $A \subseteq X$, $\overline{fA} = f \overline{A}$.
\end{lemma}

\begin{proof}
\leanok{}
    This is a simple exercise using the fact that $f$ is a closed map.
\end{proof}

\begin{lemma}
    \label{lemma_open_nbhd_of_fiber_contains_open_preimage}
    \lean{existOpenNeighborhoodPreImageContainedIn}
    \leanok{}
    Let $\pi: X \to Y$ be a surjection of compact Hausdorff spaces.  For all $y \in Y$ and all open sets $W \subseteq X$ such that $\pi^{-1} y \subseteq W$, there exists an open set $U \subseteq Y$ containing $y$ such that $\pi^{-1} U \subseteq W$.
\end{lemma}

\begin{proof}
\leanok{}
    The set $X \setminus W$ is closed and hence compact.  The continuous image of compact sets is compact, so $\pi (X \setminus W)$ is compact, hence closed, and does not contain $y$.  Set $U = Y \setminus \pi(X \setminus W)$, and check that $y \in U$ and $\pi^{-1} U \subseteq W$.
\end{proof}

It is a useful fact that factor maps of minimal $(\N,+)$-systems are semiopen \cite[Lemma 2.9]{glasscock_koutsogiannis_richter_2019}, but this does not hold for general semigroup actions.
The following lemma will serve for us as a substitute.

\begin{lemma}      
    \label{lemma_image_of_entourage_somewhere_dense_along_diagonal}
    \lean{openProductInEntourageImage}
    \leanok{}
    Let $\pi: X \to Y$ be a surjection of compact Hausdorff spaces. For all open sets $\alpha \subseteq X^2$ containing the diagonal, there exists a non-empty, open set $V \subseteq Y$ such that
    \[V \times V \subseteq (\pi \times \pi) \alpha.\]
\end{lemma}

\begin{proof}
\leanok{}
    Let $\alpha \subseteq X^2$ be open and contain $\Delta_X$. By \cref{lemma_closed_nbhd_prod_in_diag}, for all $x \in X$, there exists an open set $U_x \subseteq X$ containing $x$ such that $\overline{U_x} \times \overline{U_x} \subseteq \alpha$.

    By compactness, there exists a finite set $F \subseteq X$ such that $\bigcup_{x \in F} U_x = X$. Thus $\bigcup_{x \in F} \overline{U_x} = X$.  Since $\pi$ is surjective, $\bigcup_{x \in F} \pi \overline{U_x} = Y$. Since $\pi$ is a closed map, the sets $\pi \overline{U_x}$ are closed.  Since $Y$ is compact Hausdorff, it is a Baire space, so there exists $x \in F$ such that $\pi \overline{U_x}$ has non-empty interior.  Let $V \subseteq Y$ be non-empty, open such that $V \subseteq \pi \overline{U_x}$.  Then
    \[
        V \times V \subseteq \pi \overline{U_x} \times \pi \overline{U_x} = (\pi \times \pi) (\overline{U_x} \times \overline{U_x}) \subseteq (\pi \times \pi)\alpha,
    \]
    as was to be shown.
\end{proof}

The following result is implicit in most papers on the subject.  It is a consequence of the Alexandroff Theorem \cite[Theorem 3.2.11]{engelking_1989}.

\begin{lemma}
\label{lemma_quotient_is_compact_hausdorff}
\lean{quotientByCERIsCompactHausdorff}
\leanok{}
    Let $X$ be a compact Hausdorff space, and let $R \subseteq X^2$ be a closed equivalence relation.  The quotient $X/R$, endowed with the quotient topology, is a compact Hausdorff space.
\end{lemma}

\subsection{Semigroups}
\label{sec_semigroups}

The semigroups in this paper are all discrete and are all written multiplicatively, though we will frequently assume commutativity.  Semigroups are assumed to be non-empty but are not assumed to have an identity element.  Given a semigroup $S$ and a topological space $X$, an \emph{action of $S$ on $X$ by continuous maps} is a semigroup homomorphism from $S$ into $C(X,X)$, the set of continuous self-maps of $X$ under composition; more simply, for all $s, t \in S$ and $x \in X$, the maps $s, t, st : X \to X$ are continuous and $s(tx) = (st) x$. 

A subset $A$ of a semigroup $S$ is \emph{syndetic} if there exists a finite set $F \subseteq S$ such that $F^{-1} A = S$.  The set $A$ is \emph{thick} if for all finite $F \subseteq S$, there exists $s \in S$ such that $Fs \subseteq A$.

\begin{lemma}
    \label{lemma_synd_thick_dual}
    \lean{syndeticThickIntersect}
    \leanok{}
    Let $S$ be a semigroup.  If $A \subseteq S$ is syndetic and $B \subseteq S$ is thick, then $A \cap B \neq \emptyset$.
\end{lemma}

\begin{proof}
\leanok{}
    Since $A$ is syndetic, there exists a finite set $F \subseteq S$ such that $F^{-1} A = S$.  Since $B$ is thick, there exists $s \in S$ such that $Fs \subseteq B$.  We see $s \in S = F^{-1} A$, so there exists $f \in F$ such that $fs \in A \cap B$, as desired.
\end{proof}

\begin{lemma}
    \label{lemma_thick_right_shift_is_thick}
    \lean{rightTransOfThickIsThick}
    \leanok{}
    Let $S$ be a semigroup and $H \subseteq S$ be thick.  For all $t \in S$, the set $Ht$ is thick.
\end{lemma}

\begin{proof}
\leanok{}
    Let $F \subseteq S$ be finite and $t \in S$.  There exists $s \in S$ such that $Fs \subseteq H$.  Then $Fst \subseteq Ht$.
\end{proof}

\begin{lemma}
    \label{lemma_image_of_syndetic}
    \lean{surjImgOfSyndeticIsSyndetic}
    \leanok{}
    If $\varphi: S \to T$ is a surjective semigroup homomorphism and $A \subseteq S$ is syndetic, then $\varphi A \subseteq T$ is syndetic.
\end{lemma}

\begin{proof}
\leanok{}
    Since $A \subseteq S$ is syndetic, there exists a finite set $F \subseteq S$ such that for all $s \in S$, there exists $f \in F$ such that $fs \in A$.

    Put $G \defeq \varphi F$; it is a finite subset of $T$. We claim that for all $t \in T$, there exists $g \in G$ such that $gt \in \varphi A$.  Indeed, let $t \in T$.  Write $t = \varphi s$ for some $s \in S$.  There is $f \in F$ such that $fs \in A$.  Apply $\varphi$ to see that $gt = \varphi(f) t \in \varphi A$, as desired.
\end{proof}

We end this subsection with some combinatorial lemmas we will need later.

\begin{lemma}
\label{lemma_delta_sets_in_thick_sets}
\lean{thickSetContainsLeftDeltaSequence}
\lean{thickSetContainsRightDeltaSequence}
\leanok{}
    Let $S$ be a semigroup, let $H \subseteq S$ be thick, and let $k \in \N$.
    \begin{enumerate}
        \item There exist $s_1, s_2, \ldots, s_k \in S$ such that for all $i < j$, $s_i \in H s_j$.

        \item There exist $s_1, s_2, \ldots, s_k \in S$ such that for all $i < j$, $s_j \in s_iH$.
    \end{enumerate}
\end{lemma}

\begin{proof}
\leanok{}
    Since $H$ is thick, for all finite $F \subseteq S$, there exists $s \in S$ such that $Fs \subseteq H$. Continuing, since $Fs \cup \{s\}$ is finite, there exists $s' \in S$ such that $Fss' \cup \{ss'\} \subseteq H$.  We see that for all $F \subseteq S$ finite, there exists $s \in H$ such that $Fs \subseteq H$.
    
    We claim that there exist $t_1, \ldots, t_k \in S$ such that for all $i \leq j$, $t_i \cdots t_j \in H$.  Indeed, choose $t_1 \in H$.  Supposing $t_1, \ldots, t_n \in S$ are such that all consecutive, increasing sub-products are in $H$, put $F = \{t_1 \cdots t_n, t_2 \cdots t_n , \ldots, t_{n-1}t_n, t_n\}$.  There exists $t_{n+1} \in H$ such that $Ft_{n+1} \subseteq H$.  It is easy to see that all consecutive, increasing sub-products of $t_1, \ldots, t_{n+1}$ are in $H$.

    To see 1., put $s_i = t_i \cdots t_k$, and note that $s_i = t_{i} \cdots t_{j-1} s_j,$ where $t_i \cdots t_{j-1} \in H$.

    To see 2., put $s_i = t_1 \cdots t_i$, and note that $s_j = s_i t_{i+1} \cdots t_j,$ where $t_{i+1}\cdots t_j \in H$.
\end{proof}

\begin{lemma}
\label{lemma_syndetic_quotient_sets}
\lean{unionOfRightQuotientSetsIsSyndetic}
\lean{unionOfLeftQuotientSetsIsSyndetic}
\leanok{}
    Let $S$ be a semigroup.  For all finite partitions $S = \cup_{i=1}^k S_i$, the sets $\cup_{i=1}^k S_iS_i^{-1}$ and $\cup_{i=1}^k S_i^{-1}S_i$ are syndetic.
\end{lemma}

\begin{proof}
\leanok{}
    Let $S = \cup_{i=1}^k S_i$ be a finite partition of $S$.  We will show that the sets $\cup_{i=1}^k S_iS_i^{-1}$ and $\cup_{i=1}^k S_i^{-1}S_i$ have non-empty intersection with every thick set.  Let $H \subseteq S$ be thick.
    
    By \cref{lemma_delta_sets_in_thick_sets}, there exist $s_1, s_2, \ldots, s_{k+1} \in S$ such that for all $i < j$, $s_i \in H s_j$.  By the pigeonhole principle, there exists $m < n$ with $s_m, s_n \in S_i$.  Since $s_m \in H s_n$, there exists $h \in H$ such that $s_m = h s_n$.  Then $h \in S_i S_i^{-1}$, so $H \cap \cup_{i=1}^k S_iS_i^{-1} \neq \emptyset$, as was to be shown.

    By \cref{lemma_delta_sets_in_thick_sets}, there exist $s_1, \ldots, s_{k+1} \in S$ such that for all $i < j$, $s_j \in s_i H$.  By the pigeonhole principle, there exists $s_m, s_n \in S_i$.  Write $s_n = s_m h$ for some $h \in H$.  Then $h \in S_i^{-1} S_i$, so $H \cap \cup_{i=1}^k S_i^{-1}S_i \neq \emptyset$, as was to be shown.
\end{proof}

\subsection{Dynamics}
\label{sec_dynamics}

As in the introduction, we define a \emph{(topological dynamical) system} to be an action of a (discrete) semigroup $S$ on a compact Hausdorff space $X$ by continuous maps.  We call $X$ an \emph{$S$-system}.

The following construction appears frequently without special mention. Let $\varphi: S \to T$ be a semigroup homomorphism.
If $X$ is a $T$-system, then there is a natural way to make $X$ into an $S$-system: for $s \in S$ and $x \in X$, define $s : x \mapsto \varphi(s)x$.
Since $\varphi$ is a homomorphism, it is immediate that this yields an action of $S$ on $X$ via continuous maps.

\subsubsection{Visit times and orbits}

\label{subsec:visit_times}

Given a point $x \in X$ and sets $U, V \subseteq X$, we write
\begin{align*}
    R(x,U) &\defeq \big\{s \in S \mid sx \in U\big\} \text{ and} \\
    R(V,U) &\defeq \big\{s \in S \mid sV \cap U \neq \emptyset\big\}
\end{align*}
for the set of times that $x$ visits $U$ and $V$ visits $U$, respectively.

\begin{lemma}
\label{lemma_return_time_algebra_one}
\lean{visitsToPreimages}
\leanok{}
    Let $S$ be a semigroup, $X$ be an $S$-system, $x \in X$, $U \subseteq X$, and $s \in S$. We have
    \[s^{-1}R(x,U) = R(x,s^{-1}U).\]
\end{lemma}

\begin{proof}
\leanok{}
    This is a simple exercise from the definitions and is left to the reader.
\end{proof}

The \emph{orbit} of $x$ is $Sx$, while $\overline{Sx}$ is its \emph{orbit closure}.  The following simple lemma describes an ``orbit closure transitivity'' that is used numerous times throughout the paper.

\begin{lemma}
\label{theorem_orbit_closure_transitivity}
\lean{orbitTransitivity}
\leanok{}
    Let $X$ be an $S$-system and $x,y,z \in X$. If $y \in \overline{Sx}$ and $z \in \overline{Sy}$, then $z \in \overline{Sx}$.
\end{lemma}

\begin{proof}
\leanok{}
    Let $U \subseteq X$ be open and contain $z$.  Since $z \in \overline{Sy}$, we have $U \cap Sy \neq \emptyset$, and so there exists $s \in S$ such that $y \in s^{-1}U$.  Since $y \in \overline{Sx}$ and $s^{-1}U$ is open, $s^{-1} U \cap Sx \neq \emptyset$.  Thus, there exists $t \in S$ such that $tx \in s^{-1} U$.  Therefore, $stx \in U$, and so $U \cap Sx \neq \emptyset$.  Since $U$ was arbitrary, we get that $z \in \overline{Sx}$, as was to be shown.
\end{proof}

\subsubsection{Factors and ICERs}

As defined in the introduction, a \emph{factor} of $X$ is an $S$-system $Y$ together with a continuous surjection $\pi: X \to Y$ that intertwines the actions of $S$.  An \emph{invariant, closed equivalence relation (ICER) on $X$} is an equivalence relation on $X$ that, as a subset of $X^2$, is closed and invariant under the diagonal action of $S$.  Given an ICER $R$, the quotient $X/R$ is, by \cref{lemma_quotient_is_compact_hausdorff}, a compact Hausdorff space.  The \emph{quotient system} is obtained by defining $S$ to act by $s[x]_R \defeq [sx]_R$.

There is a 1--1 correspondence between factors (up to isomorphism) and ICERs.
Indeed, on the one hand, given an ICER $R \subseteq X^2$, the quotient $X \to X/R$ is a factor of $X$.
On the other hand, given a factor $\pi: X \to Y$, we see that
\[R_\pi \defeq \big\{ (x,y) \in X^2 \mid \pi x = \pi y \big\}\]
is an ICER and that the $S$-systems $X/R_\pi$ and $Y$ are isomorphic: the map $X / R_\pi \to Y$ given by $[x]_{R_\pi} \mapsto \pi x$ is a homeomorphism that intertwines the actions of $S$.

\subsubsection{Subsystems, minimality and uniform recurrence}
\label{sec_subsystems_and_minimality}

Given a non-empty, closed, $S$-invariant set $Z \subseteq X$, we consider $Z$ as an $S$-system by restricting the maps $s \in S$ to $Z$.  This is a \emph{subsystem} of $X$.  The most frequently occurring subsystems in this paper are of the form $\overline{Sx}$ for a point $x \in X$.

As in the introduction, an $S$-system $X$ is \emph{minimal} if for all $x \in X$, the orbit $Sx$ is dense in $X$.  Two facts follow from standard arguments:
\begin{enumerate}
    \item \label{item_min_iff_triv_subsystems}
    A system $X$ is minimal if and only if $X$ is the only subsystem of $X$.
    \item Every system $X$ has a minimal subsystem, that is, a subsystem that is a minimal system.
\end{enumerate}
If $\pi: X \to Y$ is a factor map of $S$-systems and $Z$ is a subsystem of $X$, then $\pi Z$ is a subsystem of $Y$.  It follows from this and Fact \ref{item_min_iff_triv_subsystems}. that if $Z$ is a minimal subsystem of $X$, then $\pi Z$ is a minimal subsystem of $Y$.

\begin{lemma}
\label{lemma_return_time_algebra_three}
\lean{setVisitsAsQuotientSet}
\leanok{}
    Let $S$ be a semigroup, $X$ be a minimal $S$-system, $U, V \subseteq X$ open, and $x \in X$.  We have
    \[R(x,V)R(x,U)^{-1} = R(U,V).\]
\end{lemma}

\begin{proof}
\leanok{}
    Each assertion below is quickly seen to be equivalent to the one that follows it.  Minimality is used only to show that assertion c. implies assertion d. \cref{lemma_return_time_algebra_one} is used to show that e. and f. are equivalent.
    \renewcommand{\theenumi}{\alph{enumi}}
    \begin{enumerate}
        \item $s \in R(U,V)$.
        \item $sU \cap V$ is non-empty.
        \item $U \cap s^{-1} V$ is non-empty and open.
        \item There exists $t \in S$ such that $t \in R(x, U \cap s^{-1}V)$.
        \item There exists $t \in S$ such that $t \in R(x,U) \cap R(x,s^{-1}V)$.
        \item There exists $t \in S$ such that $t \in R(x,U) \cap s^{-1} R(x,V)$.
        \item There exists $t \in R(x,U)$ such that $t \in s^{-1} R(x,V)$.
        \item There exists $t \in R(x,U)$ such that $st \in R(x,V)$.
        \item There exists $t \in R(x,U)$ such that $s \in R(x,V) t^{-1}$.
        \item $s \in R(x,V) R(x,U)^{-1}$.
    \end{enumerate}
    The conclusion of the lemma follows because the first and last assertions are equivalent.
\end{proof}

A point $x \in X$ is called \emph{$S$-uniformly recurrent} if for all open $U \subseteq X$ containing $x$, the set $R(x,U)$ is syndetic. Two facts follow from standard arguments:
\begin{enumerate}
    \setcounter{enumi}{2}
    \item \label{item_ur_implies_min}
    If the point $x$ is uniformly recurrent, then its orbit closure $\overline{Sx}$ is a minimal $S$-system.
    
    \item \label{item_minimal_implies_ur}
    All points in a minimal system are uniformly recurrent.
\end{enumerate}
The facts above will be used throughout the paper without mention, save the following lemma, which we label for clarity.

\begin{lemma}
    \label{lemma_in_orb_clos_of_ur_implies_ur}
    \lean{inOrbitClosOfURPointImpliesUR}
    \leanok{}
    Let $S$ be a semigroup, $X$ be an $S$-system, and $x, y \in X$. If $x$ is $S$-uniformly recurrent and $y \in \overline{Sx}$, then $y$ is $S$-uniformly recurrent. 
\end{lemma}

\begin{proof}
\leanok{}
    This is an immediate consequence of Facts \ref{item_ur_implies_min}. and \ref{item_minimal_implies_ur}.
\end{proof}

\begin{lemma}
    \label{lemma_forward_backward_set_orb_clos_coincide}
    \lean{forwardBackwardSetOrbClosCoincideInBronsSys}
    \leanok{}
    Let $S$ be a semigroup and $X$ be an $S$-system. Suppose that the set of $S$-uniformly recurrent points in $X^2$ is dense (the system $X$ satisfies the \emph{Bronstein condition}).  For all thick sets $A, B \subseteq S$ and all non-empty, open sets $W \subseteq X^2$,
    \[\overline{AW} = \overline{B^{-1} W}.\]
\end{lemma}

\begin{proof}
\leanok{}
    Let $W \subseteq X^2$ be an open and non-empty set and let $(x,y) \in \overline{AW}$. We will show $(x, y) \in \overline{B^{-1} W}$. For this, let $V \subseteq X^2$ be an open set containing $(x,y)$.  Since $V \cap AW \neq \emptyset$, there exists $s \in A$ such that $V \cap sW \neq \emptyset$.  Therefore, $W \cap s^{-1}V \neq \emptyset$. By assumption, there exists an $S$-uniformly recurrent point $(a,b) \in W \cap s^{-1}V$. By \cref{lemma_thick_right_shift_is_thick}, the set $Bs$ is thick. Since $(a,b)$ is $S$-uniformly recurrent, by \cref{lemma_synd_thick_dual}, we have $(a,b) \in \overline{Bs(a,b)}$.  Therefore, there exists $t \in B$ such that $ts(a,b) \in W$.  We see $s(a,b) \in V \cap t^{-1} W$, whereby $V \cap B^{-1} W \neq \emptyset$.  Since $V$ was an arbitrary open set containing $(x,y)$, we have that $(x,y) \in \overline{B^{-1} W}$.

    Conversely, let $(x,y) \in \overline{B^{-1}W}$, and let $V \subseteq X^2$ be an open set containing $(x,y)$.  Since $V \cap B^{-1}W \neq \emptyset$, there exists $s \in B$ such that $V \cap s^{-1}W \neq \emptyset$. By assumption, there exists an $S$-uniformly recurrent point $(a,b) \in V \cap s^{-1}W$. Since the set $As$ is thick (by \cref{lemma_thick_right_shift_is_thick}) and $(a,b)$ is $S$-uniformly recurrent, $(a,b) \in \overline{As(a,b)}$.  Therefore, there exists $t \in A$ such that $ts(a,b) \in V \cap tW$, whereby $tW \cap V \neq \emptyset$.  Therefore, $V \cap A W \neq \emptyset$.  Since $V$ was an arbitrary open set containing $(x,y)$, we have that $(x,y) \in \overline{A W}$.
\end{proof}

\subsubsection{Proximality, distality, and semisimplicity}

Points $x$ and $y$ in an $S$-system $X$ are \emph{proximal} if for every open set $\alpha \subseteq X^2$ containing the diagonal, there exists $s \in S$ such that $s(x,y) \in \alpha$.  We state the following simple lemmas for future reference.

\begin{lemma}
    \label{lemma_prox_pair_under_factor_map_is_prox}
    \lean{imageOfProxByFactorIsProx}
    \leanok{}
    Let $S$ be a semigroup and $\pi: X \to Y$ be a factor map of $S$-systems.  If $x, y \in X$ are proximal, then $\pi x, \pi y \in Y$ are proximal.
\end{lemma}

\begin{proof}
    \leanok{}
    This is an easy exercise left to the reader.
\end{proof}

\begin{lemma}
    \label{theorem_proximal_and_minimal_implies_orbit_closure_contains_diagonal}
    \lean{minSystemOrbitClosProxPairContainsDiag}
    \leanok{}
    Let $S$ be a semigroup, $X$ be a minimal $S$-system, and $x, y \in X$.  If $x$ and $y$ are proximal, then $\Delta_X$ is contained in the orbit closure of $(x,y)$ under $S$.
\end{lemma}

\begin{proof}
\leanok{}
    This is an easy exercise left to the reader.
\end{proof}

A system $X$ is called \emph{distal} if no distinct points are proximal, and it is called \emph{semisimple} if every point of $X$ is $S$-uniformly recurrent.  While the product of semisimple systems is rarely semisimple, the product of distal systems is easily seen to be distal.  The following lemma is well-known in the group action setting.  We make use of a future result (\cref{lemma_distal_implies_semisimple}) since its proof is done efficiently with ultrafilters.

\begin{lemma}[{cf. \cite[Prop.~2.5]{dai_xiao_2017}}]
\label{lemma_distal_and_semisimple}
\lean{distalIffDiagSemisimple}
    Let $S$ be a semigroup and $X$ be an $S$-system.  The system $X$ is distal if and only if the system $X^2$ is semisimple.
\end{lemma}

\begin{proof}
    Suppose $X$ is distal.  Since products of distal systems are distal, the system $X^2$ is distal.  We show in \cref{lemma_distal_implies_semisimple} below that distal systems are semisimple.
    
    Conversely, suppose $X^2$ is semisimple.  To show that $X$ is distal, we will show that if $x, y \in X$ are proximal, then $x = y$. Let $(x,y) \in X^2$ be a proximal pair.  There exists $w \in X$ such that $(w,w) \in \overline{S(x,y)}$.  Since $\overline{S(x,y)}$ is a minimal system, $(x,y) \in \overline{S(w,w)} \subseteq \Delta$.  Therefore, $x = y$, as was to be shown.
\end{proof}

\subsubsection{Equicontinuity and backward equicontinuity}
\label{sec_back_and_forward_equi}

As in the introduction, we define an $S$-system $X$ to be \emph{(uniformly) equicontinuous} if the set of maps $\{s: X \to X \mid s \in S\}$ is a (uniformly) equicontinuous family: for all open sets $\alpha \subseteq X^2$ containing the diagonal, there exists an open set $\beta \subseteq X^2$ containing the diagonal such that $S \beta \subseteq \alpha$.  It is quick to check that subsystems and arbitrary products of equicontinuous systems are equicontinuous.

While the results in this section are known -- we indicate more precisely below where each can be found -- we include our proofs, which avoid enveloping semigroups of $S$-systems and their idempotents, Ellis's joint continuity theorem, and the strategy of generating a group from $S$.
As such, the semisimplicity of $X^2$ is obtained directly from the pigeonhole principle (\cref{lemma_syndetic_quotient_sets}) without first establishing distality, and the link between forward and backward equicontinuity is demonstrated by appealing to the general result in \cref{lemma_forward_backward_set_orb_clos_coincide}.

In the proofs of \cref{thm_equi_surject_implies_semisimple,thm_backequi_implies_semisimple} below, we make use of the following elementary facts about compact uniform spaces $X$.
\begin{enumerate}
    \item Given $x \in X$ and an open set $U \subseteq X$ containing $x$, there exists an open set $\alpha \subseteq X^2$ containing the diagonal such that for all $y \in X$, if $(x,y) \in \alpha$, then $y \in U$.  Indeed, it is quick to check that the set
    \[\alpha = X^2 \setminus \big(\{x\} \times (X \setminus U) \big)\]
    satisfies the stipulated property.

    \item Given $\alpha \subseteq X^2$ an open set containing the diagonal, there exists a finite cover $X = \cup_i U_i$ of $X$ into non-empty, open sets such that for all $i$ and all $x, y \in U_i$, $(x,y) \in \alpha$. To do this, note that for all $w \in X$, there exists $V_w \subseteq X$ open and containing $w$ such that $V_w \times V_w \subseteq \alpha$.  By compactness, there exists a finite set $F \subseteq X$ such that $X = \cup_{w \in F} V_w$.
\end{enumerate}

\begin{theorem}[{cf. \cite[Lemma 2.9]{auslander_dai_2019}, \cite[Thm. 1.5]{dai_xiao_2017}}]
\label{thm_equi_surject_implies_semisimple}
\lean{equiSurjectiveSystemsAreSemisimple}
\leanok{}
    Let $S$ be a semigroup and $X$ be an $S$-system.  If $X$ is equicontinuous and $S$ acts on $X$ by surjections, then $X$ is semisimple.
\end{theorem}

\begin{proof}
\leanok{}
    Suppose that $X$ is equicontinuous and that $S$ acts on $X$ by surjections.  We must show that all points of $X$ are $S$-uniformly recurrent.  Let $x \in X$ and $W \subseteq X$ be an open set containing $x$.  We will show that the set $R(x,W)$ is syndetic.

    By the note prior to the statement of the theorem, there exists an open set $\alpha \subseteq X^2$ containing the diagonal such that for all $w \in X$, if $(x,w) \in \alpha$, then $w \in W$.
    Let $\alpha' \subseteq X^2$ be an open, symmetric set containing the diagonal such that $\alpha' \circ \alpha' \circ \alpha' \subseteq \alpha$.
    By equicontinuity, there exists $\gamma \subseteq X^2$ an open set containing the diagonal such that $S \gamma \subseteq \alpha'$.

    By the note prior to the statement of the theorem, there exists a finite cover $X = \cup_{i \in F} U_i$ of $X$ by open sets with the property that if $a, b \in U_i$, then $(a,b) \in \gamma$.  For each $i \in F$, fix $w_i \in U_i$.
    Define $\varphi: S \to F^{F}$ so that for all $s \in S$ and $i \in F$, $s w_i \in U_{\varphi_s i}$.  Let $S = \cup_{j} S_j$ be the finite partition of $S$ into the level sets of $\varphi$. 

    We claim that $\cup_j S_j S_j^{-1} \subseteq R(x,W)$.  Indeed, suppose $t \in S_j S_j^{-1}$ so that there is $s \in S_j$ such that $ts \in S_j$.  By surjectivity of $s$, write $x = sy$.  Then $y \in U_i$ for some $i \in F$, and so $(y, w_i) \in \gamma \subseteq \alpha'$.  By equicontinuity, $(sy, sw_i), (ts y, ts w_i) \in \alpha'$.  Since $s$ and $ts$ have the same $\varphi$-value, we have that $(s w_i, ts w_i) \in \gamma \subseteq \alpha'$.  Composing and using symmetry, we have that $(s y, ts y) \in \alpha$.  Since $sy = x$, we see that $(x, tx) \in \alpha$, whereby $tx \in W$, whereby $t \in R(x,W)$, as desired.

    By \cref{lemma_syndetic_quotient_sets}, the set $\cup_{j} S_j S_j^{-1}$ is syndetic, hence the set $R(x,W)$ is syndetic, as was to be shown.
\end{proof}

Later on, it will be convenient to define a system to be \emph{backward equicontinuous} (this is sometimes called \emph{uniformly distal}) if for all open sets $\alpha \subseteq X^2$ containing the diagonal, there exists an open set $\beta \subseteq X^2$ containing the diagonal such that $S^{-1} \beta \subseteq \alpha$.  It is quick to check that products of backward equicontinuous systems are backward equicontinuous.

\begin{lemma}[{cf. \cite[Thm 1.14]{auslander_dai_2019}}]
\label{lemma_backward_eq_implies_surjective}
\lean{backEqImpliesSurjective}
    Let $S$ be a semigroup and $X$ be an $S$-system.  If $X$ is backward equicontinuous, then $S$ acts on $X$ by surjections.
\end{lemma}

\begin{proof}
\leanok{}
    Let $s \in S$ and $x \in X$.  We will show that $x \in sX$ by showing that for all open $\alpha \subseteq X^2$ containing the diagonal, $(\{x\} \times sX) \cap \alpha \neq \emptyset$.

    Let $\alpha \subseteq X^2$ be open and contain the diagonal.  By backward equicontinuity, we get an open set $\beta \subseteq X^2$ containing the diagonal such that $S^{-1} \beta \subseteq \alpha$.  By the note prior to the statement of \cref{thm_equi_surject_implies_semisimple}, there exists a finite cover $X = \cup_{i} U_i$ by open sets so that $w, z \in U_i$ implies $(w,z) \in \beta$.
    
    By the pigeonhole principle, there exists $m < n$ and $i$ such that $s^mx, s^nx \in U_i$.  Then $s^m (x, s^{n-m} x) \in \beta$.  By backward equicontinuity, $(x, s s^{n-m-1}x) \in S^{-1}\beta \subseteq \alpha$.  Since $s s^{n-m-1}x \in sX$, we see that $(\{x\} \times sX) \cap \alpha \neq \emptyset$, as desired.
\end{proof}

The next result follows by combining results in either of \cite{auslander_dai_2019} and \cite{dai_xiao_2017} individually; we believe the direct route we take here does not appear elsewhere.

\begin{theorem}
\label{thm_backequi_implies_semisimple}
\lean{backEquiSystemsAreSemisimple}
\leanok{}
    Let $S$ be a semigroup and $X$ be an $S$-system.  If $X$ is backward equicontinuous, then $X$ is semisimple.
\end{theorem}

\begin{proof}
\leanok{}
    We will show that all points of $X$ are $S$-uniformly recurrent.  Let $x \in X$ and $W \subseteq X$ be an open set containing $x$.  We will show that the set $R(x,W)$ is syndetic.

    By the note prior to the statement of \cref{thm_equi_surject_implies_semisimple}, there exists an open set $\alpha \subseteq X^2$ containing the diagonal such that for all $w \in X$, if $(x,w) \in \alpha$, then $w \in W$.
    By backward equicontinuity, there exists $\gamma \subseteq X^2$ an open set containing the diagonal such that $S^{-1} \gamma \subseteq \alpha$.  Define $\beta = (\alpha \cap \gamma) \cup S^{-1} \gamma$.  We see that $\beta$ contains an open set containing the diagonal.  We claim that $S^{-1} \beta \subseteq \beta$.  Indeed,
    \[S^{-1} \beta \subseteq S^{-1} \gamma \cup S^{-1} S^{-1} \gamma \subseteq S^{-1} \gamma \subseteq \beta.\]

    By the note prior to the statement of \cref{thm_equi_surject_implies_semisimple}, there exists a finite cover $X = \cup_{i} U_i$ of $X$ by open sets so that if $a, b \in U_i$, then $(a,b) \in \beta$. Define $S_i \defeq R(x,U_i)$, and note that $S = \cup_{i} S_i$.  We claim that $\cup_{i=1}^k S_i^{-1} S_i \subseteq R(x,W)$.  Indeed, suppose $t \in S_i^{-1} S_i$ so that there exists $s \in S_i$ for which $st \in S_i$.  By the definition of $S_i$, $sx, stx \in U_i$.  By the property of $U_i$, we get that $s(x,tx) \in \beta$.  Since $S^{-1} \beta \subseteq \beta$, we get that $(x,tx) \in \beta \subseteq \alpha$.  By the property of $\alpha$, $tx \in W$, implying $t \in R(x,W)$, as desired.

    By \cref{lemma_syndetic_quotient_sets}, the set $\cup_{i} S_i^{-1} S_i$ is syndetic, hence the set $R(x,W)$ is syndetic, as was to be shown.
\end{proof}

\begin{lemma}[{cf. \cite[Lemma 1.9]{auslander_dai_2019}}]
\label{lemma_easy_char_of_equi_and_backward_equi}
\lean{equiSystemIffRPMInDiag}
\lean{backEquiSystemIffRPInDiag}
    Let $S$ be a semigroup and $X$ be an $S$-system.  We have that
    \begin{align*}
        \text{$X$ is equicontinuous} \qquad &\Longleftrightarrow \qquad \bigcap_{\alpha} \overline{S \alpha} \ \subseteq \ \Delta, \text{ and that}\\
        \text{$X$ is backward equicontinuous} \qquad &\Longleftrightarrow \qquad \bigcap_{\alpha} \overline{S^{-1} \alpha} \ \subseteq \ \Delta,
    \end{align*}
    where the intersections are over all open sets $\alpha \subseteq X^2$ containing the diagonal.
\end{lemma}

\begin{proof}
\leanok{}
    Suppose $X$ is an equicontinuous $S$-system.  To see that $\cap_{\alpha} \overline{S \alpha} \subseteq \Delta$, let $\gamma \subseteq X^2$ be open and contain $\Delta$.  By regularity, there exists $\beta \subseteq X^2$ open and containing $\Delta$ such that $\overline{\beta} \subseteq \gamma$.  By the definition of equicontinuity, there exists $\alpha \subseteq X^2$ open and containing $\Delta$ such that $S \alpha \subseteq \beta$.  We see $\overline{S \alpha} \subseteq \overline{\beta} \subseteq \gamma$, as was to be shown.

    Conversely, suppose $\cap_{\alpha} \overline{S \alpha} \subseteq \Delta$.  To see that $X$ is equicontinuous, let $\alpha \subseteq X^2$ be an open set containing the diagonal. We have
    \[
        \bigcap_{\gamma} \overline{S \gamma} \subseteq \Delta \subseteq \alpha,
    \]
    where the intersection is over all open sets $\gamma \subseteq X^2$ containing $\Delta$. 
    It follows that 
    \[
        X^2 \setminus \alpha \subseteq \bigcup_{\gamma} X^2 \setminus \overline{S \gamma}.
    \]
    Since $X^2 \setminus \alpha$ is compact, there exists a finite set $F$ of open subsets of $X^2$ containing $\Delta$ such that
    \[
        X^2 \setminus \alpha \subseteq \bigcup_{\gamma \in F} X^2 \setminus \overline{S \gamma},
    \]
    whereby $\bigcap_{\gamma \in F} \overline{S \gamma} \subseteq \alpha$.
    Let $\beta = \bigcap_{\gamma \in F} \gamma$. 
    Then $\beta \subseteq X^2$ an open set containing the diagonal and $S \beta \subseteq \bigcap_{\gamma \in F} \overline{S \gamma} \subseteq \alpha$, as was to be shown.

    For the characterization of backward equicontinuity, the same argument works by replacing $S$ with $S^{-1}$.
\end{proof}

In the following theorem, note that both conditions 1. and 2. hold under products, and hence the conclusion holds immediately for all products of the system $X$ with itself.

\begin{theorem}[{cf. \cite[Thm. 1.14]{auslander_dai_2019}, \cite[Thm.~1.1.5]{cao_dai_2022}}]
\label{thm_main_equi_backward_equi_thm}
\lean{equiSurjectiveIffBackEqui}
\lean{surjEquiOrBackEquiImpliesEquiBackEquiDistalSemisimpleHomeo}
\leanok{}
    Let $S$ be a semigroup and $X$ be an $S$-system.  The following are equivalent:
    \begin{enumerate}
        \item the system $X$ is equicontinuous and $S$ acts on $X$ by surjections;
        \item the system $X$ is backward equicontinuous.
    \end{enumerate}
    If either (equivalently, both) of these conditions hold, then $S$ acts on $X$ by homeomorphisms and $X$ is equicontinuous, backward equicontinuous, distal, and semisimple.
\end{theorem}

\begin{proof}
\leanok{}
    1. implies 2. \ Suppose $S$ acts on $X$ by surjections and $X$ is equicontinuous. 
    It is easy to see that $S$ acts on $X^2$ by surjections and $X^2$ is equicontinuous.
    By \cref{thm_equi_surject_implies_semisimple}, the system $X^2$ is semisimple. 
    By \cref{lemma_forward_backward_set_orb_clos_coincide}, for all open $\alpha \subseteq X^2$ containing the diagonal, $\overline{S \alpha} = \overline{S^{-1} \alpha}$.
    By the characterizations of forward and backward equicontinuity in \cref{lemma_easy_char_of_equi_and_backward_equi}, we see that $X$ is backward equicontinuous.

    2. implies 1. \ Suppose $X$ is backward equicontinuous. By \cref{lemma_backward_eq_implies_surjective}, we have that $S$ acts on $X$ by surjections. 
    Since products of backward equicontinuous systems are backward equicontinuous, the system $X^2$ is backward equicontinuous. By \cref{thm_backequi_implies_semisimple}, the system $X^2$ is semisimple.
    By \cref{lemma_forward_backward_set_orb_clos_coincide}, for all open $\alpha \subseteq X^2$ containing the diagonal, $\overline{S \alpha} = \overline{S^{-1} \alpha}$.
    By the characterizations of forward and backward equicontinuity in \cref{lemma_easy_char_of_equi_and_backward_equi}, we see that $X$ is equicontinuous.
    
    Now, suppose $X$ is equicontinuous and $S$ acts on $X$ by surjections.  By the first paragraph, the system $X^2$ is semisimple and $X$ is backward equicontinuous.  It follows by \cref{lemma_distal_and_semisimple} that the system $X$ is distal and semisimple.

    We have only left to show that $S$ acts on $X$ by homeomorphisms.  We have already that $S$ acts on $X$ by surjections, hence we need only show that $S$ acts on $X$ by injections.

    Let $s \in S$ and $x, y \in X$, and suppose $sx = sy$.  Since $X$ is distal and $(x,y)$ is a proximal pair, we have that $x=y$, whereby $s: X \to X$ is injective, as desired.
\end{proof}

\subsubsection{The equicontinuous structure relation}

An ICER $R \subseteq X^2$ is \emph{equicontinuous} if the corresponding factor $\pi: X \to X/R$ is equicontinuous.
Equivalently to the way it is defined in the introduction, the \emph{equicontinuous structure relation} on $X$, denoted $\seq$, is the intersection of all equicontinuous ICERs on $X$.  (Note that $X^2$ is always one such ICER, so the intersection is non-empty.)  Since an arbitrary intersection of ICERs is an ICER, we see that $\seq$ is an ICER.
The following lemma shows that $X/\seq$ is the largest equicontinuous factor of $X$.  The argument is standard but short, so we supply it.

\begin{lemma}
\label{lemma_seq_is_equi_and_largest}
\lean{equiStructureRelationIsEquiICER}
\lean{equiStructFactorIsLargestEquiFactor}
\leanok{}
    Let $S$ be a semigroup and $X$ be an $S$-system.  The equicontinuous structure relation $\seq$ is an equicontinuous ICER, and $X \to X/\seq$ is the largest equicontinuous factor of $X$.
\end{lemma}

\begin{proof}
\leanok{}
    Let $(R_i)_{i \in I}$ be the collection of equicontinuous ICERs of $X$, so that $\seq = \bigcap_{i \in I} R_i$.  The map
    $\varphi: X / \seq \to \prod_{i \in I} X/R_i$ defined by $[x]_{\seq} \mapsto \prod_{i \in I} [x]_{R_i}$ is continuous, injective, and intertwines the actions of $S$, so $X/\seq$ is isomorphic as an $S$-system to its image, a subsystem of $\prod_{i \in I} X/R_i$.  Since products and subsystems of equicontinuous systems are equicontinuous, we have that $X / \seq$ is equicontinuous.

    To see that $X \to X / \seq$ is the largest equicontinuous factor, note that if $\pi : X \to Y$ is an equicontinuous factor, then $R_\pi$ contains $\seq$.  The quotient $X / \seq \to X / R_\pi$ is a factor map of $S$-systems.  But $X / R_\pi$ is isomorphic to $Y$, hence $Y$ is a factor of $X/\seq$, as was to be shown.
\end{proof}

\subsubsection{Minimal actions of commutative semigroups}
\label{sec_min_actions}

In this section, we prove two basic preparatory results concerning minimal actions of commutative semigroups.

\begin{lemma}[{cf. \cite[Corollary 3.3]{auslander_dai_2019}}]
\label{lemma_minimal_comm_action_implies_surjective}
\lean{minimalCommActionIsSurjective}
\leanok{}
    Let $S$ be a commutative semigroup and $X$ be a minimal $S$-system.  For all $s \in S$, the map $s: X \to X$ is surjective.
\end{lemma}

\begin{proof}
\leanok{}
    Let $s \in S$.  Since $S$ is commutative, the set $sX$ is a closed, $S$-invariant subset of $X$.  By minimality, $sX = X$, showing that $s$ is surjective, as desired.
\end{proof}

\begin{lemma}[{cf. \cite[Thm.~2.3.11]{cao_dai_2022}}]
    \label{lemma_dense_S_minimal_points_in_X2}
    \lean{inMinCommSystemURPairsDense}
    \leanok{}
    Let $S$ be a commutative semigroup, and let $X$ be a minimal $S$-system.  The set of points $(x,y) \in X^2$ that are $S$-uniformly recurrent is dense in $X^2$.
\end{lemma}

\begin{proof}
\leanok{}
    Let $U \times V \subseteq X^2$ be a non-empty, open set. We will show that $U \times V$ contains an $S$-uniformly recurrent point. 
    
    Since every system has a minimal subsystem, there exists a closed, non-empty, $S$-invariant (under the diagonal action) set $Z \subseteq X^2$ which is a minimal $S$-system.  Note that the projection map on the first coordinate $\pi_1: X^2 \to X$ is an $S$-system factor map.  Therefore, $\pi_1 Z$ is a closed, non-empty, $S$-invariant subset of $X$.  Since $X$ is minimal, $\pi_1 Z = X$.  Therefore, there exists $(x,y) \in Z$ such that $x \in U$. 
    
    Since $X$ is a minimal $S$-system, there exists $s \in S$ such that $sy \in V$.  Note that the point $(x,sy) \in U \times V$.  We have only left to show that $(x,sy)$ is $S$-uniformly recurrent.
    
    Since $S$ is commutative, the map $X^2 \to X^2$ given by $(w,z) \mapsto (w,sz)$ commutes with the diagonal action of $S$.  Therefore, the image of $Z$ under this map is an $S$-system.  By the facts described in \cref{sec_subsystems_and_minimality}, the image of $Z$ under this map is a minimal $S$-system.  Since it contains the point $(x,sy)$, the point $(x,sy)$ is $S$-uniformly recurrent, as desired.
\end{proof}

\subsection{Ultrafilters}

In this paper, we use ultrafilters both on the acting (discrete) semigroup $S$ and on the phase space $X$.  On the acting semigroup, ultrafilters provide a powerful tool to study dynamical properties of a system.  On the phase space -- which is not assumed to be metric -- ultrafilters provide an alternative to the theory of nets in describing convergence.

An \emph{ultrafilter} on a set $X$ is a non-empty, upward-closed collection $\family$ of subsets of $X$ that does not contain the empty set, is partition regular (that is, for all $A, B \subseteq X$, if $A \cup B \in \family$, then $A \in \family$ or $B \in \family$), and is a filter (that is, for all $A, B \in \family$, $A \cap B \in \family$).  It will happen frequently that we push an ultrafilter forward through a map.  If $\varphi: X \to Y$ is a map, the \emph{pushforward of $\family$ under $\varphi$} is
\[\varphi\family \defeq \big\{ B \subseteq Y \mid \varphi^{-1} B \in \family \big\}.\]
It is quick to check that $\varphi \family$ is an ultrafilter on $Y$.

\subsubsection{Ultrafilters on compact Hausdorff spaces}

Let $X$ be a compact Hausdorff space, and let $\family$ be an ultrafilter on $X$.  Because $\family$ is partition regular and $X$ is compact, there exists a point $x \in X$ such that $\family$ contains all open sets containing $x$.  Because $\family$ is a filter and $X$ is Hausdorff, there can be at most one such point.  We say that the ultrafilter $\family$ \emph{converges to $x$} and we write $x = \lim \family$ for this unique limit point.

\subsubsection{Ultrafilters on discrete semigroups}

The Stone-\Cech{} compactification $\beta S$ of $S$ is a multifaceted object with many uses.  We outline in this section just the properties we need from it; our primary reference is the book of Hindman and Strauss \cite{hindman_strauss_book_2012}.

As a set, $\beta S$ consists of the ultrafilters on $S$.  It is a compact Hausdorff topological space, with topology generated by the basic open sets given by $\overline{A} \defeq \{p \in \beta S \mid A \in p \}$ for subsets $A \subseteq S$.  It is a semigroup, with multiplication defined for $p, q \in \beta S$ by
\[A \in pq \quad \Longleftrightarrow \quad \big\{s \in S \mid \{ t \in S \mid st \in A \} \in q\big\} \in p.\]
Taking the topology and semigroup structure together, it is a right-topological semigroup, meaning that for all $q \in \beta S$, the map $\beta S \to \beta S$ defined by $p \mapsto pq$ is continuous. For proofs of these facts, see \cite[Thm.'s 4.1 and 4.4]{hindman_strauss_book_2012}.

The left ideals of $\beta S$ play a particularly important role in its application to dynamics.  A subset $L \subseteq \beta S$ is a \emph{left ideal} if it is non-empty and $\beta S$-invariant: $(\beta S) L \subseteq L$.  It is a \emph{minimal left ideal} if it is a left ideal and does not properly contain a left ideal.  It is a fact \cite[Cor. 2.6]{hindman_strauss_book_2012} that every left ideal of a compact, Hausdorff, right-topological semigroup -- such as $\beta S$ -- contains a minimal left ideal.

We say an ultrafilter is \emph{minimal} if it belongs to a minimal left ideal of $\beta S$ and \emph{idempotent} if $pp = p$.  It is a fact \cite[Thm. 2.5]{hindman_strauss_book_2012} that every non-empty, compact subsemigroup of $\beta S$ contains an idempotent.  Since minimal left ideals of $\beta S$ are compact (\cite[Cor. 2.6]{hindman_strauss_book_2012}) and are subsemigroups, minimal left ideals of $\beta S$ contain an idempotent.  An ultrafilter is a \emph{minimal idempotent} if it is idempotent and minimal.

\begin{lemma}
    \label{lemma_ru_is_r_for_idempotent}
    \lean{minimalIdempotentsAreLeftIdentites}
    \leanok{}
    Let $S$ be a semigroup, $L \subseteq \beta S$ be a minimal left ideal, and $u \in L$  be idempotent.  For all $p \in L$, $pu = p$.
\end{lemma}

\begin{proof}
\leanok{}
    This follows from \cite[Thm. 1.31]{hindman_strauss_book_2012}.
\end{proof}

\subsubsection{Ultrafilters and \texorpdfstring{$S$}{S}-systems}
\label{sec_ufs_and_systems}

There is a natural way to see $\beta S$ as an $S$-system. Indeed, for $s \in S$, denote by $\delta_s \defeq \{ A \subseteq S \mid s \in A\} \in \beta S$ the principal ultrafilter at $s$.
It is simple to check from the definitions that $\delta_{s} \delta_t = \delta_{st}$ and that the map $\beta S \to \beta S$ defined by $p \mapsto \delta_s p$ is continuous.  
Since $\beta S$ is a compact Hausdorff space, with this action of $S$, we see that $\beta S$ is an $S$-system.
It is a fact \cite[Lem. 19.6]{hindman_strauss_book_2012} that the minimal subsystems of $\beta S$ are precisely the minimal left ideals of $\beta S$.

There is a natural way to go from a continuous action of $S$ on a compact Hausdorff space $X$ (that is, an $S$-system) to an action of $\beta S$ on $X$, though the action is not by continuous maps.
Indeed, the action map $S \times X \to X$ lifts to $\beta S \times X \to X$: fix $x \in X$, and lift the map $S \to X$ defined by $s \mapsto sx$ to the continuous map $\beta S \to X$ via the universal property of the Stone-\Cech{} compactification.  More concretely, for $p \in \beta S$ and $x \in X$, we define the point $px \in X$ to be the unique limit point of the pushforward of the ultrafilter $p$ on $S$ through the map $S \to X$ defined by $s \mapsto sx$.  To put this a third way,
\[px = y \quad \Longleftrightarrow \quad \forall \text{ open } U \ni y, \ R(x,U) \in p.\]
Though not obvious, it is a fact \cite[Thm. 4.5]{hindman_strauss_book_2012} that the map $\beta S \times X \to X$ just described in three equivalent ways yields an action of $\beta S$ on $X$: for all $p, q \in \beta S$ and all $x \in X$, $p (qx) = (pq) x$.

A number of useful facts follow.
\begin{enumerate}
    \item For all $x \in X$, $(\beta S) x = \overline{Sx}$.
    
    \item If $\pi: X \to Y$ is a factor map of $S$-systems, for all $p \in \beta S$, $p \circ \pi = \pi \circ p$.
    
    \item  (cf. \cite[Thm. 19.8]{hindman_strauss_book_2012}) If $X$ is a minimal $S$-system, then for all $x \in X$, the map $\beta S \to X$ given by $p \mapsto px$ is a factor map of $S$-systems. More generally, if $L$ is a minimal left ideal, then the map $L \to X$ defined by $p \mapsto px$ is a factor map of $S$-systems. 

    \item (cf. \cite[Sec. 19.3]{hindman_strauss_book_2012}) If $p \in \beta S$ is a minimal ultrafilter, then for all $x \in X$, the point $px$ is $S$-uniformly recurrent.

    \item (cf. \cite[Sec. 19.3]{hindman_strauss_book_2012}) If $p \in \beta S$ is idempotent, then for all $x \in X$, the points $x$ and $px$ are proximal.
\end{enumerate}

The final lemmas in this subsection concern the action of idempotent ultrafilters.  \cref{lemma_ultrafilter_idempotent_lift} below is a technical result that is useful in the proof of \cref{thm_rp_implies_same_minimal_component}, a precursor to \cref{mainthm_RP_is_eq_relation}.

\begin{lemma}
\label{lemma_distal_implies_semisimple}
\lean{distalImpliesDistalAndSemisimple}
\leanok{}
    Let $S$ be a semigroup and $X$ be an $S$-system.  If $X$ is distal, then every point in $X$ is $S$-uniformly recurrent.
\end{lemma}

\begin{proof}
\leanok{}
    Suppose $X$ is distal, and let $x \in X$.  Let $p \in \beta S$ be any minimal, idempotent ultrafilter. The point $px$ is uniformly recurrent and proximal to $x$.  Since $X$ is distal and $(x,px)$ is a proximal pair, it must be that $px = x$.  Therefore, $x$ is uniformly recurrent, as desired.
\end{proof}

\begin{lemma}
\label{lemma_exists_idempotent_px_is_x}
\lean{everyPointInMinFixedBySomeIdempotentUltrafilter}
\leanok{}
    Let $S$ be a semigroup and $X$ be a minimal $S$-system.  Let $L \subseteq \beta S$ be a minimal left ideal.  For all $x \in X$, there exists an idempotent $p \in L$ such that $px = x$.
\end{lemma}

\begin{proof}
\leanok{}
    This is part of the content of \cite[Thm. 19.23]{hindman_strauss_book_2012}.
\end{proof}

\begin{lemma}
\label{lemma_ultrafilter_idempotent_lift}
\lean{idempotentProductLifting}
\leanok{}
    Let $X$ be an $S$-system and $x \in X$.  Let $L \subseteq \beta S$ be a minimal left ideal and $p \in L$.  If $p x = x$, there exists $q \in L$ such that $q x = x$ and $pq$ is idempotent.
\end{lemma}

\begin{proof}
\leanok{}
    Treating $\beta S$ as an $S$-system, the minimal left ideals of $\beta S$ are precisely the minimal subsystems of $\beta S$. Thus, the minimal left ideal $L$ is a minimal $S$-system. By \cref{lemma_exists_idempotent_px_is_x}, there exists an idempotent $u \in L$ so that $up = p$.  Since $Lp = L$, there exists $q \in L$ such that $qp = u$.  Since $Lq = L$, there exists $r \in L$ such that $rq = u$. By the semigroup action of $\beta S$ on $X$,
    \[qx = q (p x) = (qp) x = ux = u (px) = (up) x = px = x.\]
    By \cref{lemma_ru_is_r_for_idempotent}, $ru = r$.  Therefore, we also have
    \[pq = (up)q = u (pq) = (rq) (pq) = r (qp) q = ruq = (ru) q = rq = u,\]
    as desired.
\end{proof}

\section{Regional proximality and equicontinuity for semigroup actions}
\label{sec_rp_for_semigp_actions}

In this section, we establish some basic facts on the regionally proximal relation for semigroup actions.
Many of these results appear scattered throughout the recent literature.  We collect them here for convenience, making specific reference to the literature where appropriate.

\subsection{The backward regionally proximal relation}
\label{sec_brp}

We begin with a technical tool that is useful for us throughout the paper.
There are two natural candidates for generalizing the regionally proximal relation from group actions to semigroup actions.  One is as it is defined in the introduction in \cref{def:rp_forward}.  The other is as defined in \cite{cao_dai_2022,daghar_2024}: the \emph{backward regionally proximal relation} is
\[
    \RP^- \defeq \bigcap_\alpha \overline{\bigcup_{s \in S} s \alpha},
\]
where the intersection is over all open sets $\alpha \subseteq X^2$ containing the diagonal.
Thus, $(x,y) \in \RP^-$ if and only if for all open sets $U \subseteq X$ containing $x$, $V \subseteq X$ containing $y$, and $\alpha \subseteq X^2$ containing the diagonal, there exists $(w,w') \in \alpha$ and $s \in S$ with $(sw,sw') \in U \times V$.

For group actions, it is immediate that $\RP$ and $\RP^-$ coincide, but a simple example (\cref{example_rp_not_bacward_rp}) shows that they need not be the same for semigroup actions, even when $S = (\N,+)$.  Nevertheless, we show in \cref{mainthm_RP_is_backward_RP} that for minimal actions of commutative semigroups, the two relations coincide.

We develop results for $\RP^-$ in parallel with those for $\RP$ in this section; some are useful in this paper, and the others are recorded for future reference.
We invite the curious reader to explore potentially weaker sufficient conditions for $\RP$ and $\RP^-$ to coincide than those assumed in \cref{mainthm_RP_is_backward_RP}.

\subsection{Basic facts on the regionally proximal relations}

Much, if not all, of the content of \cref{lemma_frp_basics} regarding $\RP$ is shown or implicit in \cite{auslander_dai_2019, dai_xiao_2017}.
We provide the short proof for completeness.

\begin{lemma}
\label{lemma_frp_basics}
\lean{RPisSymmetric}
\lean{RPInCommSemiIsInvariant}
\lean{RPisClosed}
\lean{RPisReflexive}
\lean{RPMisSymmetric}
\lean{RPMisInvariant}
\lean{RPMisClosed}
\lean{RPMisReflexiveIfNondegen}
\leanok{}
    Let $S$ be a semigroup and $X$ be an $S$-system.
    \begin{enumerate}
        \item The relation $\RP$ is reflexive, symmetric, and closed.  If $S$ is commutative, then the relation $\RP$ is $S$-invariant.

        \item The relation $\RP^-$ is symmetric, closed, and $S$-invariant.  If $\overline{SX} = X$, then the relation $\RP^-$ is reflexive.
    \end{enumerate}
\end{lemma}

\begin{proof}
\leanok{}
    (Reflexivity) Since $\Delta \subseteq S^{-1} \alpha$ for all $\alpha \subseteq X^2$ containing the diagonal, we have immediately from the definition of $\RP$ that $\Delta \subseteq \RP$.

    Now we will show that $\RP^-$ is reflexive under the assumption that $\overline{SX} = X$.  To see that $\Delta \subseteq \RP^-$, by the definition of $\RP^-$, it suffices to show: for all non-empty, open $U \subseteq X$ and all open $\alpha \subseteq X^2$ containing the diagonal, $(U \times U) \cap S\alpha \neq \emptyset$.  Let $U \subseteq X$ be non-empty and open, and let $\alpha \subseteq X^2$ be open and contain $\Delta$.  From the assumption that $\overline{SX} = X$, we get that there exists $s \in S$ and $x \in X$ such that $sx \in U$.  Therefore, $s(x,x) \in (U \times U) \cap S \alpha$, demonstrating that this set is non-empty.

    (Symmetry) Note that the involution $\Phi: (x,y) \mapsto (y,x)$ commutes with the action of $S$ on $X^2$ and sends open sets containing the diagonal to open sets containing the diagonal.  Let $\alpha \subseteq X^2$ be an open set containing the diagonal.  If $(x,y) \in \RP$, then $(x,y) \in \overline{S^{-1} \Phi \alpha}$, so by \cref{lemma_commuting_map_algebra},
    \[(y,x) = \Phi (x,y) \in \Phi \overline{S^{-1} \Phi \alpha} \subseteq \overline{S^{-1} \alpha}.\]
    Similarly, if $(x,y) \in \RP^-$, then $(x,y) \in \overline{S \Phi \alpha}$, so
    \[(y,x) = \Phi (x,y) \in \Phi \overline{S \Phi \alpha} = \overline{S \alpha}.\]
    Since $\alpha$ was arbitrary, this shows that $\RP$ and $\RP^-$ are symmetric.

    ($S$-invariance) First we will show that $\RP$ is $S$-invariant under the assumption that $S$ is commutative.  Let $\alpha \subseteq X^2$ be an open set containing the diagonal, and let $s \in S$.  Note that $\beta \defeq s^{-1} \alpha$ is an open set containing the diagonal. By \cref{lemma_image_closure_equals_closure_image} and \cref{lemma_commuting_map_algebra},
    \[s \RP = s \bigcap_{\gamma} \overline{S^{-1}\gamma} \subseteq s \overline{S^{-1} \beta} = \overline{sS^{-1} \beta} \subseteq \overline{S^{-1} s \beta} \subseteq \overline{S^{-1} \alpha}.\]
    Since $\alpha$ was arbitrary, we see $s \RP \subseteq \RP$, as desired.

    The proof that $\RP^-$ is $S$-invariant goes through without assuming that $S$ is commutative.  Indeed, the set $S\alpha$ is clearly $S$-invariant.  By \cref{lemma_image_closure_equals_closure_image},
        \[s \overline{S\alpha} = \overline{sS \alpha} \subseteq \overline{S \alpha},\]
    whereby $\overline{S\alpha}$ is $S$-invariant.  Finally, since an arbitrary intersection of $S$-invariant sets is $S$-invariant, we have that $\RP^-$ is $S$-invariant.

    (Closedness) This is immediate for both $\RP$ and $\RP^-$ since both are the intersections of closed sets.
\end{proof}

\begin{lemma}
\label{lemma_frp_iff_nbhd_tran_int_hits_diagonal}
\label{lemma_rp_iff_nbhd_tran_int_hits_diagonal}
\lean{inRPiffForwardUOrbitClosHitsDiag}
\lean{inRPMiffBackwardUOrbitClosHitsDiag}
\leanok{}
    Let $S$ be a semigroup, $X$ be an $S$-system, and $x, y \in X$.
    \begin{enumerate}
        \item \label{item_desc_of_frp}
        The point $(x,y) \in \RP$ if and only if for all open $U \subseteq X^2$ containing $(x,y)$,
    \begin{align*}
         \Delta \cap \overline{S U} \neq \emptyset.
    \end{align*}

        \item \label{item_desc_of_rp}
        The point $(x,y) \in \RP^-$ if and only if for all open $U \subseteq X^2$ containing $(x,y)$,
    \[\Delta \cap \overline{S^{-1} U} \neq \emptyset.\]
    \end{enumerate}
\end{lemma}

\begin{proof}
\leanok{}
1. \ Suppose $(x, y) \in \RP = \bigcap_{\alpha} \overline{S^{-1} \alpha}$. Suppose for the sake of a contradiction that there exists an open set $U \subseteq X^2$ containing $(x, y)$ such that $\Delta \cap \overline{S U} = \emptyset$. The set $\alpha = X^2 \setminus \overline{S U}$ is open, contains $\Delta$, and satisfies $SU \cap \alpha = \emptyset$. Since $(x, y) \in \RP$, we have $(x, y) \in \overline{S^{-1} \alpha}$. It follows that $U \cap S^{-1} \alpha \neq \emptyset$ or equivalently, $SU \cap \alpha \neq \emptyset$, which is a contradiction.

Conversely, let $\alpha \subseteq X^2$ be an open set containing $\Delta$. To show that $(x, y) \in \overline {S^{-1} \alpha}$, we will show that $U \cap S^{-1} \alpha \neq \emptyset$ for all open $U \subseteq X^2$ containing $(x,y)$. Let $U$ be an open set containing $(x, y)$. Since $\Delta \cap \overline{S U} \neq \emptyset$, we have that $\alpha \cap SU \neq \emptyset$, whereby $U \cap S^{-1} \alpha \neq \emptyset$, as desired.

2. \ The proof is similar to the proof of \cref{item_desc_of_frp}, using the definition of $\RP^-$ in \cref{sec_brp}.
\end{proof}

The following theorem follows from the relative form for extensions given by Cao and Dai in \cite[Lemma 1.1.5(b)]{cao_dai_2022}.

\begin{theorem}
\label{mainthm_RP_is_backward_RP}
\lean{forwardEqualsBackwardRPInMinCommSystem}
    Let $S$ be a commutative semigroup and $X$ be a minimal $S$-system.  The regionally proximal relation and the backward regionally proximal relation coincide: $\RP = \RP^-$.
\end{theorem}

\begin{proof}
    Since $S$ is commutative and $X$ is minimal, \cref{lemma_dense_S_minimal_points_in_X2} gives that the set of $S$-uniformly recurrent points of $X^2$ is dense.
    It follows from \cref{lemma_forward_backward_set_orb_clos_coincide} that for all non-empty, open $U \subseteq X^2$,
    \[\overline{S U} = \overline{S^{-1}U}.\]
    Now $\RP = \RP^-$ follows immediately from the descriptions of $\RP$ and $\RP^-$ in \cref{lemma_frp_iff_nbhd_tran_int_hits_diagonal,lemma_rp_iff_nbhd_tran_int_hits_diagonal}.
\end{proof}

It is not the case that $\RP$ and $\RP^-$ always coincide, as the following example shows.

\begin{example}[{cf. \cite[Ex.~1.1.5d]{cao_dai_2022}}]
\label{example_backwardrp_not_rp}
\label{example_rp_not_bacward_rp}
    Let $X = [-1,1]$ and $f: X \to X$ be defined by
    \[f(x) = \operatorname{sign}(x) x^2.\]
    If $S = (\N,+)$ acts on $X$ by repeated applications of $f$, one can check that
    \[
        \RP^- = [-1,0]^2 \cup [0,1]^2 \subsetneq X^2 = \RP.
    \]
    If, on the other hand, $S = (\N,+)$ acts on $X$ by repeated applications of $f^{-1}$, then $\RP$ and $\RP^-$ reverse roles, so that $\RP \subsetneq \RP^-$.
\end{example}

The following lemma offers a characterization of $\RP^-$ that will be useful in the proof of \cref{mainthm_RP_is_eq_relation}.

\begin{lemma}
\label{lemma_rp_implies_ultra}
\lean{xyInRPMIffUltraToSomewwxy}
\leanok{}
    Let $S$ be a semigroup, $X$ be an $S$-system, and $x, y \in X$. Define $\varphi: X^2 \times S \to X^4$ by
    \[\varphi (a, b, s) = (a, b, sa, sb).\]
    The point $(x,y) \in \RP^-$ if and only if there exist $w \in X$ and an ultrafilter $\family$ on $X^2 \times S$ such that $\lim \varphi \family = (w,w,x,y)$.
\end{lemma}

\begin{proof}
\leanok{}
    Suppose there exist $w \in X$ and an ultrafilter $\family$ on $X^2 \times S$ such that $\lim \varphi \family = (w,w,x,y)$.  Let $U, V \subseteq X$ and $\alpha \subseteq X^2$ be open sets containing $x, y$ and the diagonal $\Delta$, respectively. Since $\lim \varphi \family = (w,w,x,y)$ and $\alpha \times U \times V$ is an open set containing $(w,w,x,y)$, we have that $\varphi^{-1} (\alpha \times U \times V) \in \family$.  Since $\family$ is an ultrafilter, its members are non-empty.  In particular, there exists $(a,b,s) \in \varphi^{-1} (\alpha \times U \times V)$.  In other words, $(a,b,sa,sb) \in \alpha \times U \times V$, showing that $(U \times V) \cap s \alpha \neq \emptyset$.  By definition, it follows that $(x,y) \in \RP^-$.

    Conversely, suppose $(x,y) \in \RP^-$.  Denote by $N(x)$ the set of open subsets of $X$ containing $x$, by $N(y)$ the set of open subsets of $X$ containing $y$, and by $N(\Delta)$ the set of open subsets of $X^2$ containing the diagonal.  By the definition of $\RP^-$, for all $(U,V,\alpha) \in N(x) \times N(y) \times N(\Delta)$, there exists $(a, b, s) \in X^2 \times S$ such that $(a,b) \in \alpha$ and $s (a,b) \in U \times V$. This gives us a map $\psi$ from $N(x) \times N(y) \times N(\Delta)$ to $X^2 \times S$ defined by $\psi: (U,V,\alpha) \mapsto (a,b,s)$.
    
    The sets $N(x)$, $N(y)$, and $N(\Delta)$ are directed by inclusion, so $N(x) \times N(y) \times N(\Delta)$ is directed.  Let $\mathcal{G}$ be an ultrafilter extending a cofinal filter on $N(x) \times N(y) \times N(\Delta)$.  We see that $\family \defeq \psi \mathcal{G}$ is an ultrafilter on $X^2 \times S$.  Because $\mathcal{G}$ extends a cofinal filter,
    \[\lim \varphi \family \in \Delta \times \{x\} \times \{y\},\]
    and hence there exists $w \in X$ such that $\lim \varphi \family = (w,w,x,y)$, as desired.
\end{proof}

\subsection{Regional proximality in factors}

The goal in this subsection is to prove \cref{lemma_rp_functorial_property_minimal}, which says roughly that the image of $\RP_X$ under a factor map $\pi: X \to Y$ is $\RP_Y$.  To begin, we show under no assumptions that $(\pi \times \pi) \RP_X \subseteq \RP_Y$.  In the group-action setting, a modern reference is \cite[V. 2.12 and 2.14, pg. 416]{devries_book_1993}.

\begin{lemma}
    \label{lemma_frp_functorial_property}
    \label{lemma_rp_functorial_property}
    \lean{imageOfRPIsInRP}
    \lean{imageOfRPMIsInRPM}
    \leanok{}
    Let $X$ be an $S$-system and $\pi: X \to Y$ be a factor map.
    \begin{enumerate}
        \item \label{item_frp_half_functorial}
        $(\pi \times \pi) \RP_X \subseteq \RP_Y$, and 

        \item \label{item_rp_half_functorial}
        $(\pi \times \pi) \RP^-_X \subseteq \RP^-_Y$.
    \end{enumerate}
\end{lemma}

\begin{proof}
\leanok{}
    Let $\beta \subseteq Y^2$ be an open set containing $\Delta_Y$.  Since $\pi: X \to Y$ is continuous and $X$ is compact, $\pi$ is uniformly continuous.  In particular, there exists an open set $\alpha \subseteq X^2$ containing $\Delta_X$ such that $(\pi \times \pi) \alpha \subseteq \beta$.  We then see by  \cref{lemma_commuting_map_algebra} and \cref{lemma_image_closure_equals_closure_image} that
    \begin{align*}
        (\pi \times \pi) \RP_X &\subseteq (\pi \times \pi) \overline{S^{-1} \alpha} \subseteq \overline{(\pi \times \pi) S^{-1} \alpha} \subseteq \overline{S^{-1}(\pi \times \pi) \alpha} \subseteq \overline{S^{-1} \beta}, 
        \text{ and}\\
        (\pi \times \pi) \RP^-_X &\subseteq (\pi \times \pi) \overline{S \alpha} \subseteq \overline{(\pi \times \pi) S \alpha} = \overline{S(\pi \times \pi) \alpha} \subseteq \overline{S \beta},
    \end{align*}
    as was to be shown.
\end{proof}

\begin{theorem}
\label{lemma_frp_iff_nbhd_tran_int_contains_diagonal}
\lean{inMinCommxyInRPIffNhdOrbitClosContainsDiag}
\leanok{}
    Let $S$ be a commutative semigroup, $X$ be a minimal $S$-system, and $x, y \in X$.  The point $(x,y) \in \RP$ if and only if for all open $U \subseteq X^2$ containing $(x,y)$,
    \[\Delta \subseteq \overline{S U}.\]
    The three other statements gotten by independently replacing $S$ by $S^{-1}$ and $\RP$ by $\RP^-$ also hold.
\end{theorem}

\begin{proof}
\leanok{}
    Note that the ``if'' statement follows immediately from \cref{lemma_frp_iff_nbhd_tran_int_hits_diagonal}.
    
    To see the ``only if'' direction, let $U \subseteq X^2$ be an open set containing $(x,y)$.  Define
    \[Z \defeq \overline{S U}.\]
    We will show that $Z \cap \Delta$ is a subsystem of $\Delta$, that is, a non-empty, closed, and $S$-invariant subset of $\Delta$. That $Z \cap \Delta$ is closed is immediate. That $Z \cap \Delta$ is non-empty follows from \cref{lemma_frp_iff_nbhd_tran_int_hits_diagonal}.
    To see that $Z \cap \Delta$ is $S$-invariant, let $s \in S$.  Since $\Delta$ is $S$-invariant, it suffices to show that $Z$ is $S$-invariant. By our assumptions and \cref{lemma_dense_S_minimal_points_in_X2}, the $S$-system $X^2$ has a dense set of uniformly recurrent points.  By \cref{lemma_thick_right_shift_is_thick}, the set $Ss$ is thick, so by \cref{lemma_forward_backward_set_orb_clos_coincide}, $\overline{SsU} = \overline{SU}$. Therefore, by \cref{lemma_image_closure_equals_closure_image},
    \[
        sZ = s \overline{S U} = \overline{sS U} = \overline{Ss U} = Z.
    \]

    Since $X$ is a minimal $S$-system, the diagonal $\Delta$ is a minimal $S$-system. Since $Z \cap \Delta$ is a subsystem of $\Delta$, we get that $Z \cap \Delta = \Delta$. Therefore, $\Delta \subseteq Z$, as was to be shown.

    To see that the three other statements gotten by independently replacing $S$ by $S^{-1}$ and $\RP$ by $\RP^-$ also hold, note that by \cref{lemma_forward_backward_set_orb_clos_coincide}, for all non-empty, open $U \subseteq X^2$,
    \[\overline{S U} = \overline{S^{-1}U}.\]
    Also, by \cref{mainthm_RP_is_backward_RP}, we have $\RP = \RP^-$.  Therefore, given the assumptions of the theorem, all four statements are equivalent.
\end{proof}

The following theorem says that regional proximality lifts through factor maps; it is a key fact about the $\RP$ relation.
For minimal group actions, factor maps are semiopen, and the lifting follows from a short argument (see \cite[V.~2.14]{devries_book_1993}).
For semigroup actions, the result was proved by Cao and Dai \cite[Thm.~3.1.2]{cao_dai_2022} for minimal actions of arbitrary semigroups, and in the relative setting; see also \cite[Thm.~2.7]{dai_2019}, which treats the case that $S$ acts on $X$ by surjections and on $Y$ by homeomorphisms.
Our proof avoids the regularity of the universal minimal system and the invertibility of $Y$, and it uses \cref{lemma_image_of_entourage_somewhere_dense_along_diagonal} to get around the fact that $\pi$ need not be semiopen.
The price of these more elementary techniques is that we assume $S$ is commutative, an assumption Cao and Dai do not need.

\begin{theorem}
\label{lemma_rp_functorial_property_minimal}
\lean{commMinRPIsInImageOfRP}
\lean{imageOfRPIsInRP}
\lean{imageOfRPMIsInRPM}
\lean{}
    Let $S$ be a commutative semigroup, $X$ be a minimal $S$-system, and $\pi: X \to Y$ be a factor map of $S$-systems.  Then $\RP_Y = (\pi \times \pi) \RP_X$. The same statement holds with $\RP$ replaced by $\RP^-$.
\end{theorem}

\begin{proof}
\leanok{}
    By \cref{lemma_frp_functorial_property}, $(\pi \times \pi) \RP_X \subseteq \RP_Y$, so we only need to prove the reverse inclusion.  To see that $\RP_Y \subseteq (\pi \times \pi) \RP_X$, we will prove that $\RP_Y^- \subseteq (\pi \times \pi) \RP^-_X$ by showing that $Y^2 \setminus (\pi \times \pi) \RP^-_X \subseteq Y^2 \setminus \RP_Y^-$. This suffices since $\RP_X = \RP_X^-$ and $\RP_Y = \RP_Y^-$ by \cref{mainthm_RP_is_backward_RP}.
    
    Suppose $(x,y) \in Y^2 \setminus (\pi \times \pi) \RP^-_X$.  For all $z \in (\pi \times \pi)^{-1} \{(x,y)\}$, we have that $z \not\in \RP^-_X$.  By the definition of $\RP^-_X$, there exists an open set $V_{z} \subseteq X^2$ containing $z$ and an open set $\alpha_z \subseteq X^2$ containing $\Delta_X$ such that
    \[S^{-1}V_{z} \cap \alpha_z = \emptyset.\]

    Now
    \[
        (\pi \times \pi)^{-1} \{(x,y)\} \subseteq \bigcup_{z \in (\pi \times \pi)^{-1} \{(x,y)\}} V_{z},
    \]
    so by compactness, there exists a finite set $F \subseteq (\pi \times \pi)^{-1} \{(x,y)\}$ such that
    \[
        (\pi \times \pi)^{-1} \{(x,y)\} \subseteq \bigcup_{z \in F} V_{z}.
    \]
    We have by \cref{lemma_open_nbhd_of_fiber_contains_open_preimage} (with $(x,y)$ as $y$ and $\bigcup_{z \in F} V_{z}$ as $W$) that there exists an open set $U \subseteq Y^2$ containing $(x,y)$ such that $(\pi \times \pi)^{-1} U \subseteq \bigcup_{z \in F} V_{z}$. 

    Put $\beta = \bigcap_{z \in F} \alpha_z$; it is an open set containing $\Delta_X$. We see that
    \begin{align*}
        \emptyset = \bigcup_{z \in F} S^{-1}V_z \cap \beta  &= \big( S^{-1} \bigcup_{z \in F} V_z \big) \cap \beta \\
        &\supseteq \big(S^{-1} (\pi \times \pi)^{-1} U \big) \cap \beta \\
        & = \big((\pi \times \pi)^{-1} S^{-1} U \big) \cap \beta,
    \end{align*} 
    whereby $\big((\pi \times \pi)^{-1} S^{-1} U \big) \cap \beta = \emptyset$, or, equivalently,
    \[
        \big(S^{-1} U \big) \cap (\pi \times \pi) \beta  = \emptyset.
    \]
    
    By \cref{lemma_image_of_entourage_somewhere_dense_along_diagonal}, there exists a non-empty, open set $Z \subseteq Y$ such that $Z \times Z \subseteq (\pi \times \pi) \beta$, so
    \[
        \big(S^{-1} U \big) \cap (Z \times Z)  = \emptyset.
    \]
    It follows that $\Delta \not\subseteq \overline{S^{-1} U}.$
    Since $U$ is an open subset of $Y^2$ containing $(x,y)$, we see by \cref{lemma_frp_iff_nbhd_tran_int_contains_diagonal} that $(x,y) \not\in \RP_Y^-$, as was to be shown.
\end{proof}

\subsection{Regional proximality and equicontinuity}
\label{sec_rp_equicontint_and_proof_of_a}

A modern reference for \cref{lemma_rp_and_equicontinuity} below in the group-action setting is \cite[pg. 283]{devries_book_1993}.  The statement in 1. below is proven in \cite[Cor. 1.16]{auslander_dai_2019} and \cite[Thm. 1.5]{dai_xiao_2017}.  Our take in this paper appears to be different, so we provide it here.

\begin{lemma}
    \label{lemma_rp_and_equicontinuity}
    \lean{equiSystemIffRPMInDiag}
    \lean{equicontinuousIffRPTrivialAndSurjective}
    Let $S$ be a semigroup and $X$ be an $S$-system.
    \begin{enumerate}
        \item \label{item_eq_iff_rp_in_delta}
        $\RP \subseteq \Delta$ if and only if $X$ is equicontinuous and $S$ acts on $X$ by surjections.
        
        \item \label{item_eq_iff_rpminus_in_delta}
        $\RP^- \subseteq \Delta$ if and only if $X$ is equicontinuous.
    \end{enumerate}
\end{lemma}

\begin{proof}
\leanok{}
    To prove 1., note that combining the definition of $\RP$ from \cref{sec_rp_and_seq} with the characterization of backward equicontinuity in \cref{lemma_easy_char_of_equi_and_backward_equi}, it suffices to prove that $X$ is backward equicontinuous if and only if $X$ is equicontinuous and $S$ acts by surjections.  This is precisely the content of \cref{thm_main_equi_backward_equi_thm}.
    
    The statement in 2. follows immediately by combining the definition of $\RP^-$ from \cref{sec_brp} with the characterization of equicontinuity in \cref{lemma_easy_char_of_equi_and_backward_equi}.
\end{proof}

\begin{theorem}[{cf. \cite[Corollary 3.1.2c]{cao_dai_2022}}]
    \label{lemma_rp_characterization_of_equicont}
    \label{lemma_rp_characterization_of_equicont_icer}
    \lean{minimalFactorEquicontinuousIffRPInFactorRelation}
    \lean{minimalICEREquicontinuousIffRPInICER}
    Let $S$ be a commutative semigroup and $X$ be a minimal $S$-system.
    \begin{enumerate}
        \item A factor $\pi: X \to Y$ is equicontinuous if and only if $\RP_X \subseteq R_\pi$.
        
        \item An ICER $Q$ of $X$ is equicontinuous if and only if $\RP_X \subseteq Q$.
    \end{enumerate}
    The same statement holds with $\RP_X$ replaced by $\RP_X^-$.
\end{theorem}

\begin{proof}
\leanok{}
    Assertions 1. and 2. are equivalent.  Indeed, by definition, an ICER $Q$ is equicontinuous if and only if the factor $X \to X/Q$ is equicontinuous, and $R_{X \to X/Q} = Q$.  Moreover, by \cref{mainthm_RP_is_backward_RP}, we have $\RP_X = \RP_X^-$.  Therefore, it suffices to prove assertion 1. for $\RP_X^-$.

    Let $\pi: X \to Y$ be a factor. Note that by the definition of $R_\pi$, we have $\RP_X^- \subseteq R_\pi$ if and only if $(\pi \times \pi) \RP_X^- \subseteq \Delta_Y$.  By \cref{lemma_rp_functorial_property_minimal}, this happens if and only if $\RP^-_Y \subseteq \Delta_Y$.  By \cref{lemma_rp_and_equicontinuity}, this happens if and only if $Y$ is equicontinuous.
\end{proof}

\section{Corner space}
\label{sec_corner_space}

In this section, we set up the main dynamical tool behind the proof of \cref{mainthm_RP_is_eq_relation}, the corner space.  It is a simple instance of the Host-Kra dynamical cubes that are the modern standard for capturing nilpotent structure in dynamical systems.  Throughout this section, let $S$ be a commutative semigroup and $X$ be an $S$-system.

\subsection{Action of \texorpdfstring{$S^2$}{S2} on \texorpdfstring{$X^3$}{X3}}
\label{sec:action_s2_x3}

For $s, t \in S$, define maps $T_{L,s}, T_{R,t}: X^3 \to X^3$ by
\begin{align*}
    T_{L,s}(x,y,z) &= (sx, sy, z)\\
    T_{R,t}(x,y,z) &= (x, ty, tz).
\end{align*}
Since $S$ is commutative, the map $S^2 \times X^3 \to X^3$ given by
    \[(s,t)(x,y,z) \defeq T_{L,s}T_{R,t}(x,y,z) = (sx,sty,tz)\]
makes $X^3$ an $S^2$-system.

In the following lemma, the action of $S$ on $X^2$ is the diagonal action $s (x, y) = (sx, sy)$ while the action of $S^2$ on $X^3$ is the one just described.

\begin{lemma}
    \label{lemma_S_orb_closure_contained_in_S2_orbit_closure_v2}
    \label{lemma_S_orb_closure_contained_in_S2_orbit_closure_v3}
    \lean{liftTo3Right}
    \lean{liftTo3Left}
    \leanok{}
    If $(x,y,z) \in X^3$ is $S^2$-uniformly recurrent, then
    \begin{enumerate}
        \item \label{item_lift_one}
        $\overline{S (x,y)} \times \{z\} \subseteq \overline{S^2(x,y,z)}$, and
        \item \label{item_lift_two}
        $\{x\} \times \overline{S (y,z)} \subseteq \overline{S^2(x,y,z)}$.
    \end{enumerate}
\end{lemma}

\begin{proof}
\leanok{}
    Suppose $(x,y,z) \in X^3$ is $S^2$-uniformly recurrent.
    To see \ref{item_lift_one}., we will show that for all $s \in S$, $(sx,sy,z) \in \overline{S^2(x,y,z)}$.  Let $s \in S$, and let $W \subseteq X^3$ be an open set containing $(sx,sy,z)$. The map $T_{L,s}: X^3 \to X^3$ is continuous, so $T_{L,s}^{-1}W$ is an open set containing $(x,y,z)$. Since $(x,y,z)$ is uniformly recurrent, there is $(a,b) \in S^2$ such that $(a,b)(x,y,z) \in T_{L,s}^{-1}W$.  Therefore, $(as,b)(x,y,z) \in W$.  Therefore, $W \cap \overline{S^2(x,y,z)} \neq \emptyset$.  Since $W$ was arbitrary, we see that $(sx,sy,z) \in \overline{S^2(x,y,z)}$.

    The proof of \ref{item_lift_two}. is nearly identical, using $T_{R,s}$ instead of $T_{L,s}$. 
\end{proof}

\subsection{Action of \texorpdfstring{$S^2$}{S2} on \texorpdfstring{$X^2$}{X2}}

Overloading the notation, for $s, t \in S$, we define maps $T_{L,s}, T_{R,t}: X^2 \to X^2$ by
\begin{align*}
    T_{L,s}(x,y) &= (sx, sy)\\
    T_{R,t}(x,y) &= (x, ty).
\end{align*}
Note that these maps are precisely the restrictions of the previously-defined $T_{L,s}$ and $T_{R,s}$ to the first two coordinates.
Since $S$ is commutative, the map $S^2 \times X^2 \to X^2$ given by
    \[(s,t)(x,y) \defeq T_{L,s}T_{R,t}(x,y) = (sx,sty)\]
makes $X^2$ an $S^2$-system. Note that $X^2$ is now both an $S$-system (via the diagonal action) and an $S^2$-system and that these actions commute.

\begin{lemma}
    \label{lemma_x3_to_x2_is_S2_factor_map}
    \lean{pi12FromCorner3ToCorner2IsFactor}
    \leanok{}
    The projection-to-the-first-two-coordinates map $\pi_{12}: X^3 \to X^2$ is a factor map of $S^2$-systems.
\end{lemma}

\begin{proof}
\leanok{}
    Since $\pi_{12}$ is clearly a continuous surjection, we need only show that for all $(s,t) \in S^2$, $(s,t) \pi_{12} = \pi_{12} (s,t)$ as maps from $X^3$ to $X^2$.  Let $(s,t) \in S^2$ and $(x,y,z) \in X^3$.  We see
    \[
        (s,t) \pi_{12} (x,y,z) = (s,t)(x,y) = (sx,sty) = \pi_{12} (sx, sty, tz) = \pi_{12} (s,t) (x,y,z),
    \]
    as desired.
\end{proof}

\begin{lemma}
    \label{lem_minimal_h_from_g}
    \lean{minimalSystemImpliesMinimalcornerSystem2}
    \leanok{}
    If $X$ is minimal as an $S$-system, then $X^2$ is minimal as an $S^2$-system.
\end{lemma}

\begin{proof}
\leanok{}
    We will show that every point in $X^2$ has a dense $S^2$-orbit.  Let $(x,y) \in X^2$ and let $U, V \subseteq X$ be nonempty, open sets. Since $X$ is a minimal $S$-system, there exists $s \in R_S(x,U)$ so that $(sx,sy) \in U \times X$.  Since $X$ is a minimal $S$-system, there exists $t \in R_S(sy,V)$ so that
    \[(s,t) (x,y) = (sx,tsy) \in U \times V,\]
    as was to be shown.
\end{proof}

\subsection{Action of \texorpdfstring{$S^2$}{S2} on \texorpdfstring{$X$}{X}}

Overloading the notation again, for $s, t \in S$, we define maps $T_{L,s}, T_{R,t}: X \to X$ by
\begin{align*}
    T_{L,s}z &= z\\
    T_{R,t}z &= tz.
\end{align*}
Observe that these maps are the restrictions of the maps in \cref{sec:action_s2_x3} to the third coordinate.
The map $S^2 \times X \to X$ given by
\[
    (s,t)z \defeq T_{L,s}T_{R,t}z = tz
\]
makes $X$ an $S^2$-system.  Note that $X$ is now both an $S$-system (via the original action) and an $S^2$-system and that these actions commute.

\begin{lemma}
\label{lemma_x3_to_x_is_S2_factor_map}
\lean{pi3FromCorner3ToCorner1IsFactor}
\leanok{}
    The projection-to-the-third-coordinate map $\pi_3: X^3 \to X$ is a factor map of $S^2$-systems.
\end{lemma}

\begin{proof}
\leanok{}
    Since $\pi_3$ is clearly a continuous surjection, we need only show that for all $(s,t) \in S^2$, $(s,t) \pi_{3} = \pi_{3} (s,t)$ as maps from $X^3$ to $X$. Let $(s,t) \in S^2$ and $(x,y,z) \in X^3$.  We see
    \[
        (s,t) \pi_{3} (x,y,z) = (s,t)z = tz = \pi_3 (sx, sty, tz) = \pi_{3} (s,t) (x,y,z),
    \]
    as desired.
\end{proof}

\begin{lemma}
    \label{lemma_prox_in_h_implies_in_g}
    \lean{cornerProxImpliesProx}
    \leanok{}
    If points $x, y \in X$ are $S^2$-proximal, then they are $S$-proximal.
\end{lemma}

\begin{proof}
\leanok{}
    Suppose $x, y \in X$ are $S^2$-proximal.  Let $\alpha \subseteq X^2$ be an open set containing the diagonal.  There exists $(s,t) \in S^2$ such that $((s,t)x,(s,t)y) \in \alpha$.  But $((s,t)x,(s,t)y) = (tx,ty) = t(x,y)$, implying the $S$-orbit of $(x,y)$ under the diagonal action intersects $\alpha$. Since $\alpha$ was arbitrary, this shows that $x$ and $y$ are $S$-proximal.
\end{proof}

\subsection{Ultrafilters and the corner space}
 
In the previous subsections, we described $X^3$, $X^2$, and $X$ as $S^2$-systems.
These actions of $S^2$ allow us to describe a semigroup action of $\beta S^2$ on $X^3$, $X^2$, and $X$ as outlined in \cref{sec_ufs_and_systems}.
For $p \in \beta S^2$, we use juxtaposition to denote this action.
Thus, we write $p(x,y,z)$, $p(x,y)$, and $pz$ to denote the action of $p$ on $(x,y,z) \in X^3$, $(x,y) \in X^2$, and $z \in X$.

\begin{lemma}
\label{lemma_ultrafilter_action_on_X3}
\lean{ultraActionOnX3}
\leanok{}
    For all $p \in \beta S^2$ and all $(x,y,z) \in X^3$,
    \begin{align}
    \label{eqn_ultrafilter_action_on_X3}
        p (x,y,z) = (p (x,y), p z).
    \end{align}
\end{lemma}

\begin{proof}
\leanok{}
    Since $\pi_{12}: X^3 \to X^2$ and $\pi_3: X^3 \to X$ are factor maps of $S^2$-systems (\cref{lemma_x3_to_x2_is_S2_factor_map} and \cref{lemma_x3_to_x_is_S2_factor_map}), and since the ultrafilter action commutes with factor maps, we get that 
    \[  \pi_{12} p (x,y,z) = p \pi_{12} (x,y,z) = p(x,y)
    \]
    and 
    \[
        \pi_3 p (x,y,z) = p \pi_3 (x,y,z) = pz,
    \]
    as claimed.
\end{proof}

\section{Corner space dynamics}
\label{sec_corner_space_dynamics}

In this section, we prove \cref{mainthm_RP_is_eq_relation}.  Throughout, $S$ is a commutative semigroup and $X$ is a minimal $S$-system.  Recall from \cref{sec_corner_space} how we treat $X^3$, $X^2$, and $X$ as $S^2$-systems and the ensuing actions of $\beta S^2$ on $X^3$, $X^2$, and $X$.

\subsection{\texorpdfstring{$S^2$}{S2}-uniformly recurrent points in \texorpdfstring{$X^3$}{X3}}

\begin{theorem}
\label{thm_xxx_is_trtl_unif_rec_v3}
\label{thm_xyy_is_trtl_unif_rec_v3}
\label{thm_xxy_is_trtl_unif_rec_v3}
\lean{xxxUniformlyRecurrent}
\lean{xyyUniformlyRecurrent}
\lean{xxyUniformlyRecurrent}
\leanok{}
    Let $S$ be a commutative semigroup and $X$ be a minimal $S$-system. For all $x, y \in X$, the points $(x, y, y)$ and $(x, x, y)$ are $S^2$-uniformly recurrent.
\end{theorem}

\begin{proof}
\leanok{}
    Since $X$ is a minimal $S$-system, by \cref{lem_minimal_h_from_g}, we have $X^2$ is a minimal $S^2$-system.  Let $L$ be a minimal left ideal of $\beta S^2$. By \cref{lemma_exists_idempotent_px_is_x}, there exists a minimal, idempotent $p \in L$ such that $p(x,x) = (x,x)$.  We see by \cref{lemma_ultrafilter_action_on_X3} that $p(x,x,x) = (x,x,p x)$.
    
    Since $p$ is minimal, the point $p(x,x,x) = (x,x,px)$ is $S^2$-uniformly recurrent.  Since $p$ is idempotent, the points $(x,x,x)$ and $(x,x,px)$ are $S^2$-proximal.  By \cref{lemma_x3_to_x_is_S2_factor_map}, the map $\pi_3$ is a factor map of $S^2$-systems.  By \cref{lemma_prox_pair_under_factor_map_is_prox}, the points $x$ and $px$ are $S^2$-proximal.
    
    By \cref{lemma_prox_in_h_implies_in_g}, the points $x$ and $px$ are $S$-proximal.  It follows by \cref{theorem_proximal_and_minimal_implies_orbit_closure_contains_diagonal} that $(y,y) \in \overline{S(x,px)}$. By \cref{lemma_S_orb_closure_contained_in_S2_orbit_closure_v3}, $(x,y,y) \in \overline{S^2(x,x,px)}$.  By \cref{lemma_in_orb_clos_of_ur_implies_ur}, all points in $\overline{S^2(x,x,px)}$ are $S^2$-uniformly recurrent, whereby $(x,y,y)$ is $S^2$-uniformly recurrent, as desired.

    Since $X$ is $S$-minimal, $(x,x) \in \overline{S(y,y)}$. By the work in the previous paragraphs (with $x = y$), the point $(y,y,y)$ is $S^2$-uniformly recurrent. It follows by \cref{lemma_S_orb_closure_contained_in_S2_orbit_closure_v2} that $(x,x,y)$ belongs to $\overline{S^2(y,y,y)}$. By \cref{lemma_in_orb_clos_of_ur_implies_ur}, all points in $\overline{S^2(y,y,y)}$ are $S^2$-uniformly recurrent, whereby $(x,x,y)$ is $S^2$-uniformly recurrent, as desired.
\end{proof}

\subsection{\texorpdfstring{$\RP$}{RP} is an equivalence relation}

The following theorem gives the main connection between dynamics on the corner space and the $\RP$ relation.

\begin{theorem}
\label{thm_rp_implies_same_minimal_component}
\label{theorem_cube_space_orbit_implies_rp}
\lean{xyInRPIffxyyInyyyOrbClosure}
\leanok{}
    Let $S$ be a commutative semigroup and $X$ be a minimal $S$-system.  Let $x, y \in X$.  The point $(x,y) \in \RP$ if and only if $(x,y,y) \in \overline{S^2(y,y,y)}$. The same statement holds with $\RP$ replaced by $\RP^-$.
\end{theorem}

\begin{proof}
\leanok{}
    By \cref{mainthm_RP_is_backward_RP}, $\RP = \RP^-$, so it suffices to prove the statement for $\RP^-$.

    Suppose $(x,y) \in \RP^-$. Consider the map $\varphi_1: X^2 \times S \to X^4$ given by
    \[\varphi_1(a,b,s) \defeq (a,b,sa,sb).\]
    By \cref{lemma_rp_implies_ultra}, there exists $w \in X$ and an ultrafilter $\family$ on $X^2 \times S$ such that
    \begin{align}
        \label{eqn_limit_of_varphi_one}
        \lim \varphi_1 \family = (w,w,x,y).
    \end{align}
    
    Let $L$ be a minimal left ideal of $\beta S^2$.  Since $X^2$ is a minimal $S^2$-system, for $(a,b,s) \in X^2 \times S$, there exists $p_{(a,b,s)} \in L$ such that $p_{(a,b,s)}(w,w) = (a,b)$.  Define $\varphi_2: X^2 \times S \to \beta S^2$ by
    \[\varphi_2 (a,b,s) \defeq p_{(a,b,s)},\]
    and note that $\varphi_2$ has image in $L$. Because $\varphi_2 \family$ is an ultrafilter on $\beta S^2$ and $\beta S^2$ is compact, there exists $p \in \beta S^2$ for which
    \begin{align}
        \label{eqn_limit_of_varphi_two}
        \lim \varphi_2 \family = p.
    \end{align}
    Since $L$ is closed and the image of $\varphi_2$ is in $L$, $p \in L$.
    
    Define the map $\varphi_3: \beta S^2 \to X^2$ given by
    \[\varphi_3q \defeq q(w,w).\]
    Consider the projection map onto the first two coordinates $\pi_{12}: X^4 \to X^2$ and the composition $\varphi_3 \circ \varphi_2: X^2 \times S \to X^2$.  On the one hand, by the definitions, \[\varphi_3 \circ \varphi_2 = \pi_{12} \circ \varphi_1.\]
    Therefore, we see from \eqref{eqn_limit_of_varphi_one}
    and the fact that $\pi_{12}$ is continuous that
    \[\lim (\varphi_3 \circ \varphi_2) \family = \lim (\pi_{12} \circ \varphi_1) \family = \pi_{12} \lim \varphi_1 \family = \pi_{12}(w,w,x,y) = (w,w).\]
    On the other hand, by the universal property of the Stone-\Cech{} compactification, $\varphi_3$ is continuous, so by \eqref{eqn_limit_of_varphi_two},
    \[\lim (\varphi_3 \circ \varphi_2) \family = \varphi_3 \lim \varphi_2 \family = \varphi_3 p = p(w,w).\]
    Therefore $p(w,w) = (w,w)$.

    By \cref{lemma_ultrafilter_idempotent_lift}, there exists $q \in L$ such that $q (w,w) = (w,w)$ and $pq$ is idempotent. Consider the map $\varphi_4: X^2 \times S \to X^3$ given by
    \[\varphi_4(a,b,s) \defeq (sa,sb,p_{(a,b,s)}qy) = \big((\pi_{34} \circ \varphi_1)(a,b,s),(\varphi_2 (a,b,s))qy \big).\]
    By \eqref{eqn_limit_of_varphi_one}, \eqref{eqn_limit_of_varphi_two}, and the continuity of the map $\beta S^2 \to X$ given by $r \mapsto rqy$,
    \begin{align}
        \label{eqn_limit_of_varphi_four}
        \lim \varphi_4 \family = (x,y,pqy).
    \end{align}
    
    We claim that the image of $\varphi_4$ is contained in $\overline{S^2 (w,w,y)}$.  Indeed, let $(a,b,s) \in X^2 \times S$, and note that by \cref{lemma_ultrafilter_action_on_X3} and the facts that $q(w,w) = (w,w)$ and $p_{(a,b,s)}(w,w) = (a,b)$,
    \[\varphi_4 (a,b,s) = (sa,sb,p_{(a,b,s)} q y) = T_{L,s} p_{(a,b,s)} q (w,w,y).\]
    Since $p_{(a,b,s)}q \in L \subseteq \beta S^2$, by the definition of the ultrafilter action, we have that
    \[p_{(a,b,s)} q (w,w,y) \in \overline{S^2 (w,w,y)}.\]
    Applying $T_{L,s}$, we see
    \[\varphi_4 (a,b,s) = T_{L,s} p_{(a,b)} q (w,w,y) \in T_{L,s} \overline{S^2 (w,w,y)} \subseteq \overline{S^2 (w,w,y)},\]
    as was claimed.

    Since $\overline{S^2 (w,w,y)}$ is closed and contains the image of $\varphi_4$, we have by \eqref{eqn_limit_of_varphi_four} that $(x,y,pqy) \in \overline{S^2 (w,w,y)}$.   
    Since $X$ is a minimal $S$-system, $(w,w) \in \overline{S(y,y)}$.  Since $(y,y,y)$ is $S^2$-uniformly recurrent (\cref{thm_xxx_is_trtl_unif_rec_v3}), by \cref{lemma_S_orb_closure_contained_in_S2_orbit_closure_v2}, we have that $(w,w,y) \in \overline{S^2(y,y,y)}$. By orbit-closure transitivity (\cref{theorem_orbit_closure_transitivity}), we get that
    \[(x, y, pq y) \in \overline{S^2(y,y,y)}.\]
    Note that since $(y,y,y)$ is $S^2$-uniformly recurrent (\cref{thm_xxx_is_trtl_unif_rec_v3}), by \cref{lemma_in_orb_clos_of_ur_implies_ur}, we have that $(x,y, pq y)$ is $S^2$-uniformly recurrent.
    
    Since $pq \in L \subseteq \beta S^2$ is idempotent, the point $pq y$ is $S^2$-proximal to $y$. By \cref{lemma_prox_in_h_implies_in_g}, the points $pq y$ and $y$ are $S$-proximal. It follows by \cref{theorem_proximal_and_minimal_implies_orbit_closure_contains_diagonal} that $(y,y) \in \overline{S(y,pqy)}$. Since $(x,y, pq y)$ is $S^2$-uniformly recurrent, by \cref{lemma_S_orb_closure_contained_in_S2_orbit_closure_v3}, $(x, y, y) \in \overline{S^2 (x, y, pq y)}$. By orbit-closure transitivity (\cref{theorem_orbit_closure_transitivity}), we have that
    \[(x, y, y) \in \overline{S^2 (y,y,y)},\]
    as was to be shown.\\

    Conversely, suppose $(x,y,y) \in \overline{S^2 (y,y,y)}$. To show that $(x,y) \in \RP^-$, we will show that for all open sets $\alpha \subseteq X^2$ containing the diagonal and all open $W \subseteq X^2$ containing $(x,y)$, $W \cap S \alpha \neq \emptyset$.
    
    Let $\alpha \subseteq X^2$ be open and contain the diagonal, and let $W \subseteq X^2$ be open and contain $(x,y)$. Consider the open set $W \times \{z \in X \mid (y, z) \in \alpha\}$ containing the point $(x,y,y)$. By assumption, $(x,y,y) \in \overline{S^2 (y,y,y)}$, so there exists $(s,t) \in S^2$ so that
    \[(s,t)(y,y,y) = (sy,sty,ty) \in W \times \{z \in X \mid (y, z) \in \alpha\}.\]
    That is, $(y,ty) \in \alpha$ and $s(y,ty) = (sy,sty) \in W$.  Therefore, $W \cap S \alpha \neq \emptyset$, as desired.
\end{proof}

\begin{theorem}
\label{theorem_rp_is_transitive}
\lean{minimalImpliesRPisTransitive}
\leanok{}
    Let $S$ be a commutative semigroup and $X$ be a minimal $S$-system. The regionally proximal relation, $\RP$, is transitive. The same statement holds with $\RP$ replaced by $\RP^-$.
\end{theorem}

\begin{proof}
\leanok{}
    Let $(x,y), (y,z) \in \RP$.  Since $(x,y) \in \RP$, by \cref{thm_rp_implies_same_minimal_component}, we have that $(x,y,y) \in \overline{S^2(y,y,y)}$.  Since $X$ is minimal, $(z,z) \in \overline{S(y,y)}$.  Since $(x,y,y)$ is $S^2$-uniformly recurrent (\cref{thm_xxx_is_trtl_unif_rec_v3}), by \cref{lemma_S_orb_closure_contained_in_S2_orbit_closure_v3}, $(x,z,z) \in \overline{S^2(x,y,y)}$.  Since $(y,z) \in \RP$, by \cref{thm_rp_implies_same_minimal_component}, we have that $(y,z,z) \in \overline{S^2(z,z,z)}$.  Since $X$ is minimal, $(y,y) \in \overline{S(z,z)}$.  Since $(y,z,z)$ is $S^2$-uniformly recurrent (\cref{thm_xxx_is_trtl_unif_rec_v3}), by \cref{lemma_S_orb_closure_contained_in_S2_orbit_closure_v3}, $(y,y,y) \in \overline{S^2(y,z,z)}$.
    
    Recapping the conclusions from the previous paragraph, we have
    \begin{align*}
        (y,z,z) &\in \overline{S^2(z,z,z)},\\
        (y,y,y) &\in \overline{S^2(y,z,z)},\\
        (x,y,y) &\in \overline{S^2(y,y,y)},\text{ and}\\
        (x,z,z) &\in \overline{S^2(x,y,y)}.\\
    \end{align*}
    By orbit-closure transitivity (\cref{theorem_orbit_closure_transitivity}), we see that
    \[(x,z,z) \in \overline{S^2(z,z,z)}.\]
    It follows by \cref{theorem_cube_space_orbit_implies_rp} that $(x,z) \in \RP$, as was to be shown.

    The same statement holds for $\RP^-$ because, by \cref{mainthm_RP_is_backward_RP}, we have $\RP = \RP^-$.
\end{proof}

\begin{theorem}[$\RP$ is an ICER]
\label{theorem_rp_is_ICER}
\lean{RPisICER}
\leanok{}
    Let $S$ be a commutative semigroup and $X$ be a minimal $S$-system. The regionally proximal relation, $\RP$, is an ICER. The same statement holds with $\RP$ replaced by $\RP^-$.
\end{theorem}

\begin{proof}
\leanok{}
    By \cref{mainthm_RP_is_backward_RP}, we have $\RP = \RP^-$. The only property missing, then, from \cref{lemma_frp_basics} is transitivity, which is guaranteed by \cref{theorem_rp_is_transitive}.
\end{proof}

\subsection{Proof of \texorpdfstring{\cref{mainthm_RP_is_eq_relation}}{Theorem A}}

Now we can prove \cref{mainthm_RP_is_eq_relation}.  Note that by \cref{mainthm_RP_is_backward_RP}, \cref{mainthm_RP_is_eq_relation} holds for $\RP^-$ in place of $\RP$.

\begin{proof}[Proof of \cref{mainthm_RP_is_eq_relation}]
\leanok{}
    We must show that $\RP = \seq$.  By \cref{theorem_rp_is_ICER}, $\RP$ is an ICER.  It follows from \cref{lemma_rp_characterization_of_equicont_icer} that $\RP$ is an equicontinuous ICER so that by the definition of $\seq$, $\seq \subseteq \RP$.  Since $\seq$ is an equicontinuous ICER (\cref{lemma_seq_is_equi_and_largest}), it follows from \cref{lemma_rp_characterization_of_equicont_icer} that $\RP \subseteq \seq$.  Therefore, $\RP = \seq$.

    By \cref{lemma_seq_is_equi_and_largest}, $X / \seq$ is the largest equicontinuous factor of $X$.  Since $\seq = \RP$, we have that $X / \RP$ is the largest equicontinuous factor of $X$.
\end{proof}

\subsection{A proof of \texorpdfstring{\cref{maintheorem_thm_a_refinement}}{Theorem B} and an invariance corollary}
\label{sec_proof_of_recurrence_corollary}

In this section, we give two corollaries of our approach to \cref{mainthm_RP_is_eq_relation}, the first of which was stated as \cref{maintheorem_thm_a_refinement} in the introduction. Using the notation in \cref{subsec:visit_times}, \cref{maintheorem_thm_a_refinement} can be rephrased as follows: \emph{Let $S$ be a commutative semigroup, $X$ be a minimal $S$-system, and $x, y \in X$. The point $(x,y) \in \RP$ if and only if for all non-empty, open $V, W \subseteq X$ with $y \in V$, the set $R(x, W) \cap R(V, W)$ is syndetic.}

\begin{proof}[Proof of \cref{maintheorem_thm_a_refinement}]
    It is an immediate consequence of \cref{lemma_frp_iff_nbhd_tran_int_contains_diagonal} that $(x,y) \in \RP$ if and only if for all open $U \subseteq X$ containing $x$, open $V \subseteq X$ containing $y$, and non-empty, open $W \subseteq X$,
    \[R(U,W) \cap R(V,W) \neq \emptyset.\]
    Since $R(x, W) \subseteq R(U, W)$, the ``if'' direction of \cref{maintheorem_thm_a_refinement} holds.

    To see the ``only if'' direction, suppose $(x,y) \in \RP$.  Let $V, W \subseteq X$ be non-empty, open with $y \in V$.  Let $z \in W$.  By \cref{lemma_frp_basics}, $\RP$ is symmetric, so $(y,x) \in \RP$.
    By \cref{thm_rp_implies_same_minimal_component}, $(y,x,x) \in \overline{S^2 (x,x,x)}$. Since $X$ is minimal, $(z,z) \in \overline{S(x,x)}$.  Since $(y,x,x)$ is $S^2$-uniformly recurrent (\cref{thm_xxx_is_trtl_unif_rec_v3}), by \cref{lemma_S_orb_closure_contained_in_S2_orbit_closure_v3}, $(y,z,z) \in \overline{S^2(y,x,x)}$.  By orbit-closure transitivity (\cref{theorem_orbit_closure_transitivity}), $(y,z,z) \in \overline{S^2 (x,x,x)}$.
    
    By \cref{thm_xxx_is_trtl_unif_rec_v3}, the point $(x,x,x)$ is $S^2$-uniformly recurrent.  Since $V \times W \times W$ is an open set containing $(y,z,z)$ and $(y,z,z) \in \overline{S^2 (x,x,x)}$, we get that $R_{S^2}((x,x,x), V \times W \times W)$ is syndetic. This is the set of pairs $(g,r) \in S^2$ such that $T_{L,g} T_{R,r} (x,x,x) \in V \times W \times W$ -- that is, such that $r \in R(x,W)$, $g \in R(x,V)$, $rg \in R(x,W)$.  For each such pair $(g,r) \in S^2$, we see by \cref{lemma_return_time_algebra_three} that
    \[r \in R(x,W) \cap (R(x,W) R(x,V)^{-1}) = R(x,W) \cap R(V,W).\]
    Applying the second-coordinate projection map $\pi_2: S^2 \to S$, a surjective semigroup homomorphism, to this set of pairs and using \cref{lemma_image_of_syndetic}, we see that the set $R(x,W) \cap R(V,W)$ is syndetic, as desired.
\end{proof}

The assertion that $\RP \subseteq S^{-1} \RP$
is equivalent to the assertion that $\RP$ is $S$-invariant, which was shown in \cref{lemma_frp_basics} to hold for any $S$-system under the assumption that $S$ is commutative.  In \cref{thm_rp_strong_s_inv} below, we show that $\RP = S^{-1} \RP$ under the additional assumption of minimality.  This turns out to be an important ingredient in our proof of \cref{mainthm_kronecker_is_same_as_extension_kronecker}.  We derive \cref{thm_rp_strong_s_inv} as a corollary of \cref{mainthm_RP_is_eq_relation} and the following analogue for the equicontinuous structure relation.  

\begin{lemma}
\label{thm_seq_strong_s_inv}
\lean{equiStructureRelationIsStrongSInvariant}
\leanok{}
    If $S$ is a commutative semigroup and $X$ is a minimal $S$-system, then
    \[\seq = S^{-1} \seq.\]
\end{lemma}

\begin{proof}
\leanok{}
    The relation $\seq$ is $S$-invariant by its definition since an arbitrary intersection of $S$-invariant sets is $S$-invariant.  Therefore,
    \[\seq \subseteq S^{-1}\seq.\]
    
    To see the reverse inclusion, suppose $(x,y) \not\in \seq$ so that there exists an equicontinuous factor $\pi: X \to Y$ for which $\pi x \neq \pi y$.  We must show that for all $s \in S$, $s(x,y) \not\in \seq$.
    
    By \cref{lemma_minimal_comm_action_implies_surjective}, the semigroup $S$ acts by surjections on $X$.  Let $s \in S$.  By \cref{thm_main_equi_backward_equi_thm}, the map $s: Y \to Y$ is a homeomorphism, so $s \pi x \neq s \pi y$.  Therefore, $\pi s x \neq \pi s y$.  Since $\pi: X \to Y$ is an equicontinuous factor of $X$, we have that $(sx, sy) \not\in \seq$, as desired. 
\end{proof}

\begin{theorem}
\label{thm_rp_strong_s_inv}
\lean{RPIsStrongSInvariant}
\leanok{}
    If $S$ is a commutative semigroup and $X$ is a minimal $S$-system, then
    \[\RP = S^{-1} \RP.\]
    The same statement holds with $\RP$ replaced by $\RP^-$.
\end{theorem}

\begin{proof}
\leanok{}
    By \cref{mainthm_RP_is_eq_relation}, we have $\RP = \seq$. By \cref{thm_seq_strong_s_inv}, we have $\seq = S^{-1} \seq$. The conclusion follows immediately.  The same statement holds for $\RP^-$ because, by \cref{mainthm_RP_is_backward_RP}, we have $\RP = \RP^-$.
\end{proof}

\section{From semigroup actions to group actions}
\label{sec_nat_extension}

In this section, we develop the machinery of natural extensions for commutative semigroup actions and prove Theorems \ref{mainthm_same_rp_for_comm_actions}, \ref{mainthm_nat_ext}, and \ref{mainthm_kronecker_is_same_as_extension_kronecker}.

\subsection{Commuting actions and a proof of \texorpdfstring{\cref{mainthm_same_rp_for_comm_actions}}{Theorem B}}
\label{sec_proof_of_comm_actions_thm}

In this subsection, we prove \cref{mainthm_same_rp_for_comm_actions}, that the sets of regionally proximal pairs under two minimal, commuting actions of commutative semigroups are the same. This result exists for group actions in \cite[Thm. 1.1]{shao_xu_2026}; the proof below follows their idea with necessary adaptations for semigroup actions.

\begin{proof}[Proof of \cref{mainthm_same_rp_for_comm_actions}]
\leanok{}
    By \cref{mainthm_RP_is_backward_RP}, we have $\RP_S = \RP_S^-$ and $\RP_T = \RP_T^-$, so it suffices to prove that $\RP_S^- = \RP_T^-$. By the symmetry of $S$ and $T$, it suffices to prove only that $\RP_S^- \subseteq \RP_T^-$.
    
    Suppose $(x,y) \in \RP_S^-$.  Let $\alpha \subseteq X^2$ be an open set containing the diagonal, and let $W \subseteq X^2$ be an open set containing $(x,y)$.  We must show that $T^{-1} W \cap \alpha \neq \emptyset$.
    Since $(x,y) \in \RP_S^-$, there exists $s \in S$ such that $s^{-1} W \cap \alpha \neq \emptyset$.  Let $U, V \subseteq X$ be open and non-empty sets such that $U \times V \subseteq s^{-1} W \cap \alpha$.
    Since $S$ acts on $X$ minimally, there exists $\ell \in S$ such that
    \[Y \defeq U \cap \ell^{-1} V \neq \emptyset.\]
    Since $T$ acts on $X$ minimally, there exists  $t \in T$ such that
    \[s^{-1} Y \cap t^{-1} Y \neq \emptyset.\]
    It follows by \cref{lemma_commuting_map_algebra} that
    \[
        Z \defeq Y \cap t^{-1} s Y \supseteq Y \cap s t^{-1} Y \neq \emptyset.
    \]
    We have
    \[
        Z \subseteq U \cap t^{-1} s U.
    \]
    Moreover, by \cref{lemma_commuting_map_algebra},
    \[
        \ell Z \subseteq \ell Y \cap \ell t^{-1} s Y \subseteq \ell Y \cap t^{-1} s \ell Y \subseteq V \cap t^{-1} s V.
    \]
    Therefore,
    \[\emptyset \neq Z \times \ell Z \subseteq (U \times V) \cap t^{-1} s (U \times V) \subseteq \alpha \cap t^{-1} W,\]
    as was to be shown.
\end{proof}

\subsection{The Grothendieck group}
\label{sec_gro_gp}

It is well known that a commutative, cancellative semigroup $S$ embeds into its group of quotients.
The Grothendieck group, $\gro(S)$, generalizes this by omitting the cancellativity requirement at the cost of downgrading the embedding to a semigroup homomorphism $S \to \gro(S)$.
Since $\gro(S)$ is typically considered for monoids rather than semigroups, we briefly summarize the construction and properties we need in the following theorem.

\begin{theorem}[The Grothendieck group]
\label{thm_gro_group_props}
\lean{groGroupHomIncidenceCondition}
\lean{imageOfSIsThickInGroGroup}
\leanok{}
    Let $S$ be a commutative semigroup.  There exists an abelian group $\gro(S)$ and a semigroup homomorphism $\igro: S \to \gro(S)$ with the following properties:
    \begin{enumerate}
        \item for all finite $F \subseteq \gro(S)$, there exists $s \in S$ such that $F \igro s \subseteq \igro S$;
        \item for all $s, t \in S$, $\igro s = \igro t$ if and only if there exists $r \in S$ such that $rs = rt$.
    \end{enumerate}
\end{theorem}

\begin{proof}
\leanok{}
    Define a relation $\sim$ on $S^2$ by $(a,b) \sim (c,d)$ if there exists $s \in S$ such that $ads = bcs$.  It is a quick exercise to check that $\sim$ is an equivalence relation on $S^2$.
    
    Considering $S^2$ as a semigroup with coordinate-wise multiplication, the relation $\sim$ is multiplicatively invariant: if $(a,b) \sim (c,d)$, then $(s,t) (a,b) \sim (s,t) (c,d)$.
    The quotient $S^2 / \sim$ thus inherits a semigroup structure from $S^2$ via $[(a,b)] [(c,d)] = [(a,b) (c,d)] = [(ac,bd)]$.

    Define $\gro(S) = S^2 / \sim$.  We already have that it is a semigroup.  It is commutative because $S$ is commutative.  To see that it is a group, fix $r \in S$.  The class $[(r,r)]$ is the identity and every element has an inverse: indeed, $[(a,b)] [(b,a)] = [(r,r)]$.  Moreover, the map $\igro: S \to \gro(S)$ defined by $s \mapsto [(sr, r)]$ is a semigroup homomorphism.  We will check both of the properties asserted in the statement of the theorem.

    To see property 1., let $F = \{[(a_1, b_1)], \ldots, [(a_k, b_k)]\} \subseteq \gro(S)$ be finite. Let $s = b_1 \cdots b_k \in S$, and note that for all $i \in \{1, \ldots, k\}$,
    \[[(a_i, b_i)] \igro s = [(a_i b_1 \cdots b_k r, b_i r)] = [(a_i b_1 \cdots \hat b_i \cdots b_k r, r)] \in \igro(S),\]
    where $b_1 \cdots \hat b_i \cdots b_k$ denotes the product of all of the $b_j$'s except $b_i$.

    To see property 2., let $s, t \in S$.  Note that $\igro s = \igro t$ if and only if $(sr, r) \sim (tr, r)$, which happens if and only if there exists $p \in S$ such that $sr^2p = tr^2p$.
\end{proof}

\begin{theorem}
\label{thm_igro_same_implies_elts_act_same}
\lean{surjectiveActionsFactorThroughGroGroupHom}
\leanok{}
    Let $S$ be a commutative semigroup, let $X$ be an $S$-system, and let $s, t \in S$.  If $S$ acts by surjections and $\igro s = \igro t$, then for all $x \in X$, $sx = tx$.
\end{theorem}

\begin{proof}
\leanok{}
    For $s \in S$, denote by $f_s$ the element of $C(X,X)$ given by $s$.  Suppose $S$ acts by surjections and $\igro s = \igro t$. By \cref{thm_gro_group_props}, we have that there exists $r \in S$ such that $rs = rt$.  We wish to show that $f_s = f_t$.  

    Let $x \in X$. Since $S$ acts by surjections, there exists $y \in X$ such that $ry = x$.  Then
    \[tx = try = sry = sx.\]
    Since $x \in X$ was arbitrary, we see $f_s = f_t$.
\end{proof}

\begin{theorem}
\label{thm_extend_homeo_action_to_gr_group}
\lean{homDSofGroGroupDSIsOriginalSystem}
\leanok{}
    Let $S$ be a commutative semigroup and $X$ be an $S$-system.  If $S$ acts on $X$ by homeomorphisms, then there exists an action of $\gro(S)$ on $X$ by homeomorphisms such that the original $S$-action is the action via the homomorphism $\igro: S \to \gro(S)$ (as described at the top of \cref{sec_dynamics}).
\end{theorem}

\begin{proof}
\leanok{}
    For $s \in S$, denote by $f_s$ the homeomorphism of $X$ given by the action of $s$ on $X$. Note that since $S$ is commutative and $f_s$ is a homeomorphism, for all $s, t \in S$, $f_s \circ f_t = f_t \circ f_s$ and $f_s^{-1} \circ f_t^{-1} = f_t^{-1} \circ f_s^{-1}$.
    
    We wish to define an action of $\gro(S)$ on $X$ by homeomorphisms. Let $g \in \gro(S)$. By \cref{thm_gro_group_props}, there exists $s, t \in S$ such that $\igro(s) g = \igro(t)$, that is, $g = \igro(t) \igro(s)^{-1}$. Define $f_g: X \to X$ by $f_g = f_t \circ f_s^{-1}$.  To see that this is well-defined, suppose that $s_1, s_2, t_1, t_2 \in S$ are such that
    \[g = \igro(t_1) \igro(s_1)^{-1} = \igro(t_2) \igro(s_2)^{-1}.\]
    We have that $\igro(t_1 s_2) = \igro(t_2 s_1)$, whence by \cref{thm_igro_same_implies_elts_act_same},
    \[f_{t_1} \circ f_{s_2} = f_{t_1 s_2} = f_{t_2 s_1} = f_{t_2} \circ f_{s_1}.\]
    Composing with $f_{s_1}^{-1} \circ f_{s_2}^{-1}$, we see
    \[f_{t_1} \circ f_{s_1}^{-1} = f_{t_1} \circ f_{s_2} \circ f_{s_1}^{-1} \circ f_{s_2}^{-1} = f_{t_2} \circ f_{s_1} \circ f_{s_1}^{-1} \circ f_{s_2}^{-1} = f_{t_2} \circ f_{s_2}^{-1},\]
    as was to be shown.

    We claim that $\gro(S) \times X \to X$ given by $(g,x) \mapsto f_g x$ yields an action of $\gro(S)$ on $X$ by homeomorphisms.  Each $f_g$ is a homeomorphism by its definition. To see that it is an action, we need only check that for all $g, h \in \gro(S)$, $f_g \circ f_h = f_{gh}$.  Write $g = \igro(t_1) \igro(s_1)^{-1}$ and $h = \igro(t_2) \igro(s_2)^{-1}$.  Then $gh = \igro(t_1 t_2) \igro(s_1 s_2)^{-1}$, so we see
    \[f_{gh} = f_{t_1 t_2} \circ f_{s_1s_2}^{-1} = f_{t_1} \circ f_{t_2} \circ f_{s_1}^{-1} \circ f_{s_2}^{-1} = f_{t_1} \circ f_{s_1}^{-1} \circ f_{t_2} \circ f_{s_2}^{-1} = f_g \circ f_{h},\]
    as was to be shown.
\end{proof}

\begin{corollary}
\label{cor_S_and_GS_actions}
\lean{extendedActionRPisActionRP}
\leanok{}
    Let $S$ be a commutative semigroup, and let $X$ be a minimal $S$-system.  Suppose $S$ acts by homeomorphisms, and consider $X$ as a $\gro(S)$-system by \cref{thm_extend_homeo_action_to_gr_group}. We have $\RP_S = \RP_{\gro(S)}$.
\end{corollary}

\begin{proof}
\leanok{}
    Apply \cref{mainthm_same_rp_for_comm_actions}, noting that $S$ and $\gro(S)$ are commutative semigroups whose actions commute and are minimal.
\end{proof}

\subsection{The natural extension: a proof of \texorpdfstring{\cref{mainthm_nat_ext}}{Theorem C}}
\label{sec_nat_ext}

In this section, we generalize the construction of natural extensions in topological dynamics to treat $S$-systems for commutative semigroups $S$ that act by surjections.
More precisely, \cref{lemma_nat_ext_props} details the construction of such an extension and \cref{lemma_nat_ext_univeral_prop} shows that it satisfies a universal property.
To see why the surjectivity assumption is necessary, note that if $S$ acts on $W$ by homeomorphisms and $\pi: W \to X$ is a factor map, then for any $s \in S$, the map $\pi \circ s = s \circ \pi$ will be surjective, and hence it must be that $s : X \to X$ is surjective.

Let $\gro(S)$ be the Grothendieck group of $S$ from \cref{thm_gro_group_props}, and denote its identity element by $e_{\gro(S)}$.  The space of all functions $\gro(S) \to X$, that is, the product space $X^{\gro(S)}$, is a compact Hausdorff space with the product topology.  The group $\gro(S)$ acts on $X^{\gro(S)}$ by homeomorphisms:
\[
    \forall g, h \in \gro(S), \ \forall \varphi \in X^{\gro(S)}, \ (g \varphi) (h) = \varphi(hg).
\]
We treat $X^{\gro(S)}$ as an $S$-system via the homomorphism $\igro: S \to \gro(S)$, as described at the top of \cref{sec_dynamics}.

\begin{theorem}[The natural extension]
\label{lemma_nat_ext_props}
\lean{natExtSystem}
\lean{natExtFactorMapIsFactorMap}
\leanok{}
Let $S$ be a commutative semigroup and $X$ be an $S$-system on which $S$ acts by surjections. The set
\[
    W \defeq \big\{ \varphi: \gro(S) \to X \mid \forall g \in \gro(S), \ \forall s \in S, \ \varphi((\igro s)g) = s (\varphi (g)) \big\} \subseteq X^{\gro(S)}
\]
is a non-empty, closed, $\gro(S)$-invariant subset of $X^{\gro(S)}$. Thus, $W$ is a $\gro(S)$-system, and hence, by way of the homomorphism $\igro: S \to \gro(S)$, an $S$-system. The map $\pi: W \to X$ defined by
\[\pi (\varphi) \defeq \varphi(e_{\gro(S)})\]
is a factor map of $S$-systems.
\end{theorem}

\begin{proof}
\leanok{}
    That $W$ is closed and $\pi$ is continuous follow easily from the definition of the product topology on $X^{\gro(S)}$.

    To see that $W$ is $\gro(S)$-invariant, for $\varphi \in W$ and $h \in \gro(S)$, we must show that $h \varphi \in W$.  We must check that for all $g \in \gro(S)$ and $s \in S$, we have
    \[
        (h \varphi) ((\igro s)g) = s \big((h \varphi) (g) \big).
    \]
    Let $g \in \gro(S)$ and $s \in S$.  We see
    \[
        (h \varphi) ((\igro s)g) = \varphi((\igro s)gh) = s \big(\varphi(gh)\big) = s \big((h \varphi) (g) \big),
    \]
    as desired.

    Now we show $W$ is non-empty and $\pi$ is surjective. For $F \subseteq \gro(S)$ finite, define
    \[
        W_F = \big\{ \varphi: \gro(S) \to X \mid \forall g \in F, \ \forall s \in S, \ \varphi((\igro s)g) = s \big(\varphi (g)\big) \big\} \subseteq X^{\gro(S)}.
    \]
    It is clear that $W = \bigcap_F W_F$, where the intersection is taken over all finite sets $F \subseteq \gro(S)$.  We will show that for all $x \in X$ and all finite $F \subseteq \gro(S)$, there exists $\varphi \in W_F$ so that $\varphi(e_{\gro(S)}) = x$.  This will show that $W$ is non-empty and that $\pi$ is surjective. 
    Indeed, for any $x \in X$,
    \[\big\{ W_F \cap \{\varphi \in X^{\gro(S)} \mid \varphi(e_{\gro(S)}) = x\} \mid F \subseteq \gro(S) \text{ finite} \big\}\]
    will be a family of non-empty, closed sets with the finite intersection property.  By compactness, there will be some function in every member of the family, and that function will belong to $W$ and map to $x$ under $\pi$.

    Let $F \subseteq \gro(S)$ be finite and $x \in X$.  Without loss of generality, we can assume $e_{\gro(S)} \in F$. By \cref{thm_gro_group_props}, there exists $h \in \gro(S)$ such that $F \subseteq \igro(S) h$.  In particular, there exists $s \in S$ such that $\igro(s) h = e_{\gro(S)}$.  By our standing assumptions, $s: X \to X$ is a surjection, so there exists $y \in X$ such that $sy = x$.

    Define $\varphi : \gro(S) \to X$ by
    \[
        \varphi(g) = \begin{cases} ty & \text{ if } \exists t \in S, \ g = \igro(t) h \\ y & \text{ else } \end{cases}.
    \]
    We must verify that $\varphi$ is well-defined. If there exists $t, r \in S$ such that $g = \igro(t) h = \igro(r) h$, then $\igro(t) = \igro(r)$, and so by \cref{thm_igro_same_implies_elts_act_same}, $ty = ry$.  Note that $\varphi(h) = y$.  Indeed, consider two cases: A) there is $t \in S$ such that $h = \igro(t) h$, and B) there is no such $t$.  In case A, $\igro(t) = e_{\gro(S)}$, so $\igro(t) = \igro(t^2)$. Since $t: X \to X$ is a surjection, there exists $z \in X$ such that $tz = y$.  By \cref{thm_igro_same_implies_elts_act_same}, we have that
    \[\varphi(h) = ty = t^2 z = tz = y,\]
    as claimed. In case B, by definition, $\varphi(h) = y$, and we are done.
    
    To see that $\varphi \in W_F$, suppose $g \in F$ and $t \in S$.  Write $g = \igro(r)h$ for some $r \in S$. By the definition of $\varphi$, we see that
    \[
        \varphi(\igro(t) g) = \varphi(\igro(t)\igro(r)h) = \varphi(\igro(tr)h) = (tr)y = t(ry) = t \big(\varphi(\igro(r)h)\big) = t\big(\varphi(g)\big).
    \]
    To see that $\varphi(e_{\gro(S)}) = x$, we see by the definition of $\varphi$ that
    \[\varphi(e_{\gro(S)}) = \varphi(\igro(s)h) = sy = x.\]

    We now prove the map $\pi$ intertwines with the action of $S$. (Recall that $W$ is an $S$-system by the homomorphism $\igro: S \to \gro(S)$, meaning that $s \varphi$ is defined as $\igro(s) \varphi$.)  For this, let $\varphi \in W$ and $s \in S$.  We see
    \[
        s \big(\pi (\varphi)\big) = s \big(\varphi(e_{\gro(S)})\big) = \varphi(\igro(s) e_{\gro(S)}) = (\igro(s) \varphi)(e_{\gro(S)}) = (s \varphi)(e_{\gro(S)}) = \pi (s \varphi),
    \]
    as was to be shown.
\end{proof}

\begin{theorem}
\label{lemma_nat_ext_univeral_prop}
\lean{natExtUniversalProperty}
\lean{natExtFactorMapIsIsomIfHomeoAction}
\leanok{}
Let $S$ be a commutative semigroup and $X$ be an $S$-system on which $S$ acts by surjections.  Its natural extension $\pi: W \to X$ is a smallest extension of $X$ on which $S$ acts by homeomorphisms: if $\rho: V \to X$ is an extension of $S$-systems and $S$ acts on $V$ by homeomorphisms, then there exists a factor map $\xi: V \to W$ of $S$-systems so that $\rho = \pi \circ \xi$. Moreover, if $S$ acts on $X$ by homeomorphisms, then $\pi$ is an isomorphism of $S$-systems.
\end{theorem}

\begin{proof}
\leanok{}
    Suppose $\rho: V \to X$ is a factor map of $S$-systems and that $S$ acts on $V$ by homeomorphisms.  By \cref{thm_extend_homeo_action_to_gr_group}, since $S$ acts on $V$ by homeomorphisms, we get an action of $\gro(S)$ on $V$ that extends it. Note that this means that for all $v \in V$ and $s \in S$, we have
    \begin{align}
        \label{eqn_rho_factor_consequence}
        s (\rho v) = \rho( sv) = \rho ( (\igro s) v).
    \end{align}
    Define $\xi: V \to X^{\gro(S)}$ by
    \[\xi: v \mapsto \big( \xi_v : g \mapsto \rho (g v) \big).\]
    We will show that $\xi$ is continuous, that the image of $\xi$ is in $W$, that $\xi: V \to W$ is a factor map of $S$-systems, and that $\rho = \pi \circ \xi$.

    To see that $\xi$ is continuous, consider a cylinder set of $X^{\gro(S)}$ of the form
    \[[U]_{g} \defeq \{\varphi: \gro(S) \to X \mid \varphi(g) \in U \},\]
    where $g \in \gro(S)$ and $U \subseteq X$ is open. It is quick to check that
    \[\xi^{-1} [U]_g = g^{-1} \rho^{-1} U,\]
    which is an open subset of $V$. Finite intersections of such cylinders, which generate the topology on $X^{\gro(S)}$, will therefore have open pre-images under $\xi$, demonstrating that $\xi$ is continuous.
    
    Next we show that the image of $\xi$ is in $W$. Let $v \in V$, $g \in \gro(S)$, and $s \in S$.  We must show
    \begin{align*}
        \xi_v((\igro s) g) = s ( \xi_v(g)).
    \end{align*}
    We see that
    \[\xi_v((\igro s) g) = \rho \big( ((\igro s) g) v\big) = s \big( \rho (gv) \big) = s \big( \xi_v(g) \big),\]
    where the first and third equalities follow from the definition of $\xi_v$ and the second follows from \eqref{eqn_rho_factor_consequence}.

    To see that $\xi$ surjects onto $W$, we will show that the image of $\xi$ is dense in $W$.  This suffices since the image of $V$ under $\xi$ is closed, and hence equal to $W$.  Let $\varphi \in W$ and $g_1, \ldots, g_k \in \gro(S)$.  We will show that there exists $v \in V$ such that for all $j$, $\xi_v(g_j) = \varphi(g_j)$. By \cref{thm_gro_group_props}, there exists $h \in \gro(S)$ and $s_1, \ldots, s_k \in S$ such that $g_j = \igro(s_j) h$.  Since $h: V \to V$ is a homeomorphism and $\rho: V \to X$ is surjective, there exists $v \in h^{-1} \rho^{-1} \{ \varphi(h) \}$.  It follows that $\xi_v(h) = \rho ( hv) = \varphi(h)$. Then
    \[\varphi(g_j) = \varphi((\igro s_j) h) = s_j \big(\varphi (h)\big) = s_j \big(\xi_v(h)\big) = \xi_v((\igro s_j) h) = \xi_v(g_j),\]
    as was to be shown.

    We have only left to show that $\xi$ intertwines the action of $S$ on $V$ and $W$.  Let $v \in V$ and $s \in S$.  We must show that $\xi_{sv} = s \xi_v$.  Let $g \in \gro(S)$.  We see
    \begin{gather*}
        \xi_{sv}(g) = \rho(g(sv)) = \rho ((\igro(s) (gv)) = s \big(\rho (gv)\big) \qquad \qquad \qquad \qquad \qquad \qquad \\
        \qquad \qquad \qquad \qquad \qquad \qquad = s \big(\xi_v(g)\big) = \xi_v( (\igro s) g) = \xi_v( g(\igro s)) = (s \xi_v) (g),
    \end{gather*}
    as was to be shown.

    It remains to prove that if $S$ acts by homeomorphisms, then $\pi$ is injective, and hence is an isomorphism of $S$-systems. Note that this follows immediately from the universal property.  Indeed, the universal property gives that since $S$ acts on $X$ via homeomorphisms, we have factor maps $X \to W \to X$ that compose to give the identity.  Since $X \to W$ is surjective, we see that $\pi: W \to X$ is injective, as desired.
\end{proof}

\begin{theorem}
    \label{lemma_nat_ext_minimal}
    \lean{natExtMinimalIffSystemIsMinimal}
    \leanok{}
    Let $S$ be a commutative semigroup, $X$ be an $S$-system on which $S$ acts by surjections, and $\pi: W \to X$ be its natural extension.    
    The system $X$ is a minimal $S$-system if and only if the system $W$ is a minimal $S$-system.
\end{theorem}

\begin{proof}
\leanok{}
    Since $\pi: W \to X$ is a factor map of $S$-systems, if $W$ is a minimal $S$-system, then it is immediate that $X$ is a minimal $S$-system.

    Suppose conversely that $X$ is a minimal $S$-system. Let $\varphi, \psi \in W$, and let $U \subseteq W$ be an open set containing $\psi$. We will show that there exists $t \in S$ such that $t \varphi \in U$.
    
    By the definition of product topology on $X^{\gro(S)}$, the set $U$ contains a ``cylinder set'' around $\psi$. This means there exists a finite set $F \subseteq \gro(S)$ and, for all $f \in F$, an open set $U_f \subseteq X$ containing $\psi(f)$ satisfying the following: for any $\chi \in W$, if $\chi(f) \in U_f$ for all $f \in F$, then $\chi \in U$.

    By \cref{thm_gro_group_props}, there exists $h \in \gro(S)$ such that for all $f \in F$, there exists $s_f \in S$ such that $f = \igro(s_f) h$.  Thus, for all $f \in F$, we have that
    \[s_f \big(\psi (h)\big) = \psi( \igro(s_f) h) = \psi (f) \in U_f,\]
    whereby $\psi(h) \in \bigcap_{f \in F} s_f^{-1} U_f$.

    By the minimality of $X$, there exists $t \in S$ such that $t\varphi(h) \in \bigcap_{f \in F} s_f^{-1} U_f$; that is $s_f t \varphi(h) \in U_f$ for all $f \in F$. Moreover, for $f \in F$,
    \[
        s_f t \big(\varphi(h)\big) = \varphi( \igro(s_f) \igro(t) h) = (t \varphi) (\igro(s_f) h) = (t \varphi) (f),
    \]
    whereby $(t \varphi) (f) \in U_f$.  It follows that $t \varphi \in U$, as was to be shown.
\end{proof}

\begin{proof}[Proof of \cref{mainthm_nat_ext}]
    By \cref{lemma_minimal_comm_action_implies_surjective}, the semigroup $S$ acts by surjections.  Therefore, the hypotheses required for \cref{lemma_nat_ext_props}, \cref{lemma_nat_ext_univeral_prop}, and \cref{lemma_nat_ext_minimal} are met.
\end{proof}

\subsection{\texorpdfstring{$\RP$}{RP} in the natural extension: a proof of \texorpdfstring{\cref{mainthm_kronecker_is_same_as_extension_kronecker}}{Theorem D}}
\label{sec_proof_of_same_kron}

In the statement of the following theorem, recall that \cref{mainthm_nat_ext} guarantees the existence of the natural extension for minimal actions by commutative semigroups.
The proof of \cref{thm_rp_of_natural_extension} relies crucially on \cref{thm_rp_strong_s_inv}, which we derive from \cref{mainthm_RP_is_eq_relation}.

\begin{theorem}
\label{thm_rp_of_natural_extension}
\lean{pullBackOfRPThruNatExtFactorIsRP}
\lean{maxEquiFactorOfNatExtIsoMaxEquiFactor}
\leanok{}
    Let $S$ be a commutative semigroup, $X$ be a minimal $S$-system, and $\pi: W \to X$ be its natural extension. We have that
    \[
        (\pi \times \pi)^{-1} \RP_X = \RP_W.
    \]
    As a result, the map $W/ \RP_W \to X / \RP_X$ given by 
    \[[w]_{\RP_W} \mapsto [\pi w]_{\RP_X}\]
    is well-defined and is an isomorphism of $S$-systems.
\end{theorem}

\begin{proof}
\leanok{}
    That $\RP_W \subseteq (\pi \times \pi)^{-1} \RP_X$ follows from \cref{lemma_frp_functorial_property}.  To see that $(\pi \times \pi)^{-1} \RP_X \subseteq \RP_W$, let $(\varphi, \psi) \in W^2$, and suppose $(\pi \varphi, \pi \psi) \in \RP_X$.  We must show that $(\varphi, \psi) \in \RP_W$.
    
    First, we claim that $(\varphi(g), \psi(g)) \in \RP_X$ for all $g \in \gro(S)$.  Indeed, let $g \in \gro(S)$.  By \cref{thm_gro_group_props}, there exist $s, t \in S$ such that $(\igro s)g = \igro(t)$.  Since
    \[
        (\varphi(e_{\gro(S)}), \psi(e_{\gro(S)})) = (\pi \varphi, \pi \psi) \in \RP_X,
    \]
    by the $S$-invariance of $\RP_X$ (see \cref{lemma_frp_basics}), we have that
    \[
        s(\varphi(g), \psi(g)) = (\varphi((\igro s)g), \psi((\igro s)g)) = t (\varphi(e_{\gro(S)}), \psi(e_{\gro(S)})) \in t\RP_X \subseteq \RP_X.
    \]
    By \cref{thm_rp_strong_s_inv}, $s(\varphi(g), \psi(g)) \in \RP_X$ implies that $(\varphi(g), \psi(g)) \in \RP_X$, as desired.

    Now we will show that $(\varphi, \psi) \in \RP_W$.  Let $U, V \subseteq W$ and $\alpha \subseteq W^2$ be open sets containing $\varphi$, $\psi$, and $\Delta_W$, respectively. By the definition of product topology on $W \subseteq X^{\gro(S)}$, there exists a finite set $F \subseteq \gro(S)$ and, for all $f \in F$, open sets $U_f, V_f \subseteq X, \alpha_f \subseteq X^2$ containing $\varphi(f)$, $\psi(f)$, and $\Delta_X$ respectively such that
    \begin{itemize}
        \item for $\chi \in W$, if $\chi(f) \in U_f$ for all $f \in F$, then $\chi \in U$,

        \item for $\chi \in W$, if $\chi (f) \in V_f$ for all $f \in F$, then $\chi \in V$,

        \item for $(\chi_1, \chi_2) \in W^2$, if $(\chi_1(f), \chi_2(f)) \in \alpha_f$ for all $f \in F$, then $(\chi_1, \chi_2) \in \alpha$.
    \end{itemize}
    
    By \cref{thm_gro_group_props}, there exists $h \in \gro(S)$ and, for all $f \in F$, $s_f \in S$ such that $f = (\igro s_f) h$.
    Then for all $f \in F$, we have 
    \[
        s_f \big(\varphi(h)\big) = \varphi((\igro s_f) h) = \varphi (f) \in U_f.
    \]
    It follows that $\bigcap_{f \in F} s_f^{-1} U_f$ is an open set containing $\varphi(h)$.
    Similarly, $\bigcap_{f \in F} s_f^{-1} V_f$ and $\bigcap_{f \in F} s_f^{-1} \alpha_f$ are open sets containing $\psi(h)$ and $\Delta_X$, respectively. By the first claim, $(\varphi(h), \psi(h)) \in \RP_X$, so there exists $s \in S$ such that
    \begin{align}
        \label{eqn_nonempty_subset_of_x2}
        \left( \bigcap_{f \in F} s_f^{-1} U_f \times \bigcap_{f \in F} s_f^{-1} V_f \right) \cap s^{-1} \bigcap_{f \in F} s_f^{-1} \alpha_f \neq \emptyset.
    \end{align}
    Since $\gro(S)$ acts on $W$ by homeomorphisms and $\pi$ is surjective, the map $\pi \circ h: W \to X$ defined by $\chi \mapsto \chi(h)$ is surjective.  Hence there exists $(\eta, \rho) \in W^2$ such that $(\eta(h), \rho(h))$ belongs to the set in \eqref{eqn_nonempty_subset_of_x2}. We see that for all $f \in F$,
    \[(\eta(f), \rho(f)) = s_f (\eta(h), \rho(h)) \in U_f \times V_f,\]
    whereby $(\eta,\rho) \in U \times V$. Moreover, for all $f \in F$,
    \[s (\eta(f), \rho(f)) = s s_f (\eta(h), \rho(h)) \in \alpha_f,\]
    whereby $(\eta, \rho) \in s^{-1} \alpha$.
    Therefore,
    \[(\eta, \rho) \in (U \times V) \cap s^{-1} \alpha,\]
    showing that $(\varphi,\psi) \in \RP_W$, as was to be shown.

    Finally, we will show that the map $W/ \RP_W \to X / \RP_X$ given by 
    \begin{align}
    \label{eqn_factor_quotient_map_def}
        [w]_{\RP_W} \mapsto [\pi w]_{\RP_X}
    \end{align}
    is an isomorphism of $S$-systems.  Since $(\pi \times \pi) \RP_W \subseteq \RP_X$, the surjective composition $W \to X \to X/\RP_X$ descends to the map $W / \RP_W \to X / \RP_X$ defined in \eqref{eqn_factor_quotient_map_def}, showing that the map in \eqref{eqn_factor_quotient_map_def} is well-defined, continuous (by a defining characteristic of the quotient topology), and surjective.
    It intertwines the actions of $S$ since $\RP_W$ and $\RP_X$ are $S$-invariant.

    We only have left to show that the map $W/ \RP_W \to X / \RP_X$ is injective.  Suppose $w, z \in W$ so that $[\pi w]_{\RP_X} = [\pi z]_{\RP_X}$.  Then $(\pi \times \pi) (w,z) \in \RP_X$.  Since $(\pi \times \pi)^{-1} \RP_X = \RP_W$, we have that $(w,z) \in \RP_W$, whereby $[w]_{\RP_W} = [z]_{\RP_W}$, as was to be shown.
\end{proof}

We conclude this section by proving \cref{mainthm_kronecker_is_same_as_extension_kronecker}.

\begin{proof}[Proof of \cref{mainthm_kronecker_is_same_as_extension_kronecker}]
\leanok{}
    That \eqref{item_x_kron} is isomorphic to \eqref{item_w_kron} follows from \cref{thm_rp_of_natural_extension}.
    
    To see that \eqref{item_w_kron} is isomorphic to \eqref{item_w_kron_gp}, note first that since $X$ is a minimal $S$-system,  so is $W$ (by \cref{lemma_nat_ext_minimal}). It follows that $W$ is a minimal $\gro(S)$-system. Since the actions of $S$ and $\gro(S)$ commute, \cref{mainthm_same_rp_for_comm_actions} gives that $\RP_{W,S} = \RP_{W,\gro(S)}$.  Therefore, as sets $W / \RP_{W,S} = W / \RP_{W,\gro(S)}$.  We need only show that $S$ acts on this set in the same way in both descriptions.  But this is immediate since both actions are induced from the action of $S$ on $W$.
\end{proof}

\section{A generalization to higher-order regionally proximal relations}
\label{sec_higher_order}

Recall from the introduction that for an $S$-system $X$,
\[
    \RP \defeq \bigcap_\alpha \overline{\bigcup_{s \in S} s^{-1} \alpha},
\]
where the intersection is over all open $\alpha \subseteq X^2$ containing the diagonal.
Higher-order analogues of $\RP$ were introduced and studied by Host, Kra, and Maass \cite{host_kra_maass_2010} for $\Z$-actions.
Following their definition, for commutative semigroup actions and $d \in \N$, we might define
\begin{align}
\label{eqn_higher_order_rp}
    \RP^{[d]} := \bigcap_{\alpha}\overline{\bigcup_{s_1, \ldots, s_d \in S} \bigcap_{\substack{e_1, \ldots, e_d \in \{0,1\} \\ \text{not all zero}}} (s_1^{e_1} \cdots s_d^{e_d})^{-1} \alpha},
\end{align}
where the intersection is over all open $\alpha \subseteq X^2$ containing the diagonal and $s_1^{e_1} \cdots s_d^{e_d}$ denotes the map $X \to X$ by which the product of those $s_i$ for which $e_i = 1$, taken in increasing order of $i$, acts on $X$.

Upgrading a result in \cite{host_kra_maass_2010}, Shao and Ye \cite{shao_ye_2012} proved that $\RP^{[d]}$ is an equivalence relation for minimal actions of abelian groups.  The following question is then a natural joining of this result with part of our \cref{mainthm_RP_is_eq_relation}.

\begin{question}
\label{quest_higher_order_rp_for_semigroups}
    Let $S$ be a commutative semigroup and $X$ be a minimal $S$-system.  Is $\RP^{[d]}$ an equivalence relation?
\end{question}

As was helpful for us in this paper, it may be useful in addressing \cref{quest_higher_order_rp_for_semigroups} to develop a ``backward'' variant $(\RP^{[d]})^-$ of $\RP^{[d]}$ in parallel to the one in \eqref{eqn_higher_order_rp}.  This leads to the question: under what conditions are the forward and backward variants the same?

Glasner, Gutman, and Ye \cite{glasner_gutman_ye_2018} generalized the higher-order regionally proximal relation from the one in \eqref{eqn_higher_order_rp}, which is suitable for commutative group actions, to one suitable for arbitrary group actions.
They showed that this generalized relation is an equivalence relation for all minimal group actions.
One could similarly ask for a semigroup analogue of their results.

\begin{question}
    Is there a generalization of $\RP^{[d]}$ for semigroup actions in line with the one for group actions in \cite{glasner_gutman_ye_2018} which is an equivalence relation under all minimal actions of arbitrary semigroups?  
\end{question}

\bibliographystyle{abbrv}
\bibliography{bib}

\bigskip

\bigskip
\footnotesize
\noindent
Angelina Blahodatna\\
\textsc{University of Massachusetts Lowell}\par\nopagebreak
\noindent
\href{mailto:Angelina_Blahodatna@student.uml.edu}
{\texttt{angelina{\_}blahodatna@student.uml.edu}}

\bigskip
\footnotesize
\noindent
Lauren Detmold\\
\textsc{University of Massachusetts Lowell}\par\nopagebreak
\noindent
\href{mailto:Lauren_Detmold@student.uml.edu}
{\texttt{lauren{\_}detmold@student.uml.edu}}

\bigskip
\footnotesize
\noindent
Daniel Glasscock\\
\textsc{University of Massachusetts Lowell}\par\nopagebreak
\noindent
\href{mailto:daniel_glasscock@uml.edu}
{\texttt{daniel{\_}glasscock@uml.edu}}

\bigskip
\noindent
Anh N. Le\\
\textsc{University of Denver}\par\nopagebreak
\noindent
\href{mailto:anh.n.le@du.edu}
{\texttt{anh.n.le@du.edu}}

\end{document}